\documentclass[graybox]{svmult}

\usepackage{helvet}         
\usepackage{courier}        
\usepackage{type1cm}        
\usepackage{makeidx}         
\usepackage{graphicx}        
\usepackage{multicol}        
\usepackage[bottom]{footmisc}
\usepackage{bm}
\usepackage{mathtools}

\makeindex             

\usepackage{latexsym}
\usepackage{dsfont}
\usepackage{bbm}
\usepackage{amssymb}
\usepackage{amsmath}

\newtheorem{Theorem}{Theorem}[section]
\newtheorem{Lemma}[Theorem]{Lemma}

\newtheorem{Prop}[Theorem]{Proposition}
\newtheorem{Rem}[Theorem]{Remark}
\newtheorem{Exa}[Theorem]{Example}

\def\cF{\mathcal{F}}

\def\cL{\mathcal{L}}

\def\cN{\mathcal{N}}

\def\cS{\mathcal{S}}
\def\cT{\mathcal{T}}

\def\fX{\mathfrak{X}}

\def\bbD{\mathbb{D}} 
\def\Erw{\mathbb{E}}

\def\N{\mathbb{N}}
  
\def\Prob{\mathbb{P}} 

\def\R{\mathbb{R}}      

\def\T{\mathbb{T}}

\def\Var{\mathbb{V}{\rm ar\,}}

\def\eps{\varepsilon}
\def\vph{\varphi}

\def\1{\vec{1}}
\def\3{{\ss}}

\def\eqdist{\stackrel{d}{=}}

\def\RA{\Rightarrow}

\def\wh{\widehat}

\def\ovl{\overline}

\def\interior{\textsl{int}}

\DeclareMathOperator\esssup{ess\,sup}

\allowdisplaybreaks[4] 

\usepackage[top=3cm, bottom=3cm, outer=3.5cm, inner=3cm]{geometry}

\begin{document}

\title*{Thin tails of fixed points of the nonhomogeneous smoothing transform}
\titlerunning{Thin tails of fixed points}
\author{Gerold Alsmeyer$\,^{1}$ and Piotr Dyszewski$\,^{2}$}
\institute{$^{1}$ Inst.~Math.~Statistics, Department
of Mathematics and Computer Science, University of M\"unster,
Orl\'eans-Ring 10, D-48149 M\"unster, Germany.\at
\email{gerolda@math.uni-muenster.de}\at
$^{2}$ Institute of Mathematics, University of Wroc\l{}aw, pl. Grunwaldzki
2/4, 50-384 Wroc\l{}aw, Poland.\at
\email{piotr.dyszewski@math.uni.wroc.pl}\at
The first author was partially supported by the Deutsche Forschungsgemeinschaft (SFB 878).
The second author was partially supported by the National Science Centre, Poland (Sonata Bis, grant number DEC-2014/14/E/ST1/00588).}

\maketitle

\abstract{For a given random sequence $(C,T_{1},T_{2},\ldots)$ with nonzero $C$ and a.s. finite number of nonzero $T_{k}$, the nonhomogeneous smoothing transform $\cS$ maps the law of a~real random variable $X$ to the law of $\sum_{k\ge 1}T_{k}X_{k}+C$, where $X_{1},X_{2},\ldots$ are independent copies of $X$ and also independent of $(C,T_{1},T_{2},\ldots)$. This law is a~fixed point of $\cS$ if the stochastic fixed-point equation (SFPE) $X\eqdist\sum_{k\ge 1}T_{k}X_{k}+C$ holds true, where $\eqdist$ denotes equality in law. Under suitable conditions including $\Erw C=0$ (see \eqref{eq:standing}), $\cS$ possesses a~unique fixed point within the class of centered distributions, called the canonical solution to the above SFPE because it can be obtained as a~certain martingale limit in an~associated weighted branching model. The present work provides conditions on $(C,T_{1},T_{2},\ldots)$ such that the canonical solution exhibits right and/or left Poisson tails and the abscissa of convergence of its moment generating function can be determined. As a particular application, the right tail behavior of the Quicksort distribution is found.}
\bigskip



\section{Introduction}\label{sec:introduction}

Given a sequence $(C,T_{1},T_{2},\ldots)$ of real-valued random variables with an a.s. finite number $N:=\sum_{k\ge 1}\1_{\{T_{k}\ne 0\}}$ of nonzero $T_{k}$ and a nonzero random variable $C$, we consider the associated nonhomogeneous smoothing transform
\begin{equation*}
\cS: F\ \mapsto\ \cL\left(\sum_{k\ge 1}T_{k}X_{k}+C\right)
\end{equation*}
which maps a distribution $F$ on $\R$ to the law of $\sum_{k\ge 1}T_{k}X_{k}+C$, where $X_{1},X_{2},\ldots$ are independent and identically distributed (iid) with common law $F$ and independent of $(C,T_{1},T_{2},\ldots)$. In the case when $\cS F=F$, the distribution $F$ is called a fixed point of $\cS$. In terms of random variables, this may be stated as a so-called \emph{stochastic fixed-point equation (SFPE)}, namely
\begin{equation}\label{eq:sfpe0}
X\ \eqdist\ \sum_{k\ge 1}T_{k}X_{k}+C,
\end{equation}
where $X$ is a copy of $X_{1},X_{2},\ldots$ and $\eqdist$ means equality in law.

Special instances of \eqref{eq:sfpe0} appear above all in the asymptotic analysis of objects that exhibit a~certain kind of random recursive structure like random trees, branching processes, or recursive algorithms and data structures \cite{AldBan:05,Drmota:09,NeinRusch:05}, but also in stochastic geometry \cite{PenroseWade:04,PenroseWade:06}. As a particular and prominent example, we mention the \emph{Quicksort equation}, due to R\"osler \cite{Roesler:91}, viz.
\begin{equation}\label{eq:quicksort}
X^{(qs)}\ \stackrel{d}{=}\ UX_{1}^{(qs)}+(1-U) X_{2}^{(qs)} +g(U),
\end{equation}
which characterizes, uniquely within the class of distributions on $\R$ with mean 0 and finite variance, the limit distribution (as $n\to\infty$) of the normalized number of key comparisons made by \texttt{Quicksort}, a recursive divide-and-conquer algorithm, to sort a list of $n$ distinct numbers the order of which was chosen uniformly at random. Here $U$ has a uniform distribution on $(0,1)$ and the bounded function $g:(0,1)\to\R$ is defined as
$$ g(t)\ :=\ 2t\log t+2(1-t)\log(1-t)+1. $$
R\"osler~\cite{Roesler:01} also derived the corresponding equation for a variation of the algorithm, the median-of-three version of \texttt{Quicksort}, namely the \emph{median-of-three Quicksort equation}
\begin{equation}\label{eq:mtst}
X^{(mtqs)}\ \stackrel{d}{=}\ M X_{1}^{(mtqs)} + (1-M)X_{2}^{(mtqs)} + f(M),
\end{equation}
where as before $X^{(mtqs)}$ characterizes the limit of the normalized number of required key comparisons, $f:(0,1)\to\R$ is defined as 
$$ f(m)\ :=\ 1+ \frac{12}{7}(m\log(m) + (1-m)\log(1-m)) $$
and $M=\textsl{med}(U_{1},U_{2},U_{3})$ equals the median of three independent uniform $(0,1)$ random variables $U_{1},U_{2}$ and $U_{3}$. The latter appears because the median-of-three version of \texttt{Quicksort} chooses the partitioning element (pivot) of a sublist in each division step as the median of a small random sample (here of size 3). This makes for a more balanced partitioning at the cost of computing the median. Our results will show that $X^{(mtqs)}$ has thinner tails than $X^{(sq)}$, which confirms that for large $n$ the median-of-three version of \texttt{Quicksort} is less vulnerable to the randomness of the input than its classical counterpart. For details see Section~\ref{sec:examples}, notably \eqref{eq:qsrighttail} and \eqref{eq:qsleft} as opposed to \eqref{eq:mtqsright} and \eqref{eq:mtqsleft}, respectively.

As a third example, also from the analysis of algorithms, we mention is the \emph{2-dimensional quad tree equation} obtained by Neininger and R\"uschendorf \cite{NeinRusch:99}, viz.
\begin{align}
\begin{split}\label{eq:tdq}
X^{(qt)}\ \stackrel{d}{=}\ U_{1}U_{2} X_{1}^{(qt)} &+ U_{1}(1-U_{2})X_{2}^{(qt)}\\
&+ (1-U_{1})U_{2} X_{3}^{(qt)} + (1-U_{1})(1-U_{2}) X_4^{(qt)} + h(U_{1},U_{2}),
\end{split}
\end{align}
where $h : (0,1)^{2} \to \R$ is defined as
\begin{align*}
h(u_{1},u_{2})\ :=\ &1+u_{1}u_{2}\log(u_{1}u_{2})+(1-u_{1})u_{2} \log((1-u_{1})u_{2})\\ 
&+ u_{1}(1-u_{2})\log(u_{1}(1-u_{2})) + (1-u_{1})(1-u_{2})\log((1-u_{1})(1-u_{2}))
\end{align*}
and $U_{1},\,U_{2}$ are again independent uniform $(0,1)$ variables. Here the distribution of $X^{(qt)}$ characterizes the limit of the normalized so-called internal path length in a random quad tree, which is a data structure used to store and retrieve data from a multidimensional data set. For a formal definition of a~quad tree and its internal path length see \cite{NeinRusch:99}. We will return to the previous examples in Section~\ref{sec:examples} so as to illustrate some of our results.

\vspace{.2cm}
In the case when $N=1$, the SFPE \eqref{eq:sfpe0} takes the simple form
\begin{equation}\label{eq:rde}
X\ \eqdist\ T_{1}X+C,
\end{equation}
called \emph{random difference equation}. Random variables satisfying this equation, nowadays known as \emph{perpetuities} due to a special interpretation in the context of Mathematical Finance, appear in various quite different areas like number theory, combinatorics, branching processes in random environment, or additive-increase multiplicative-decrease (AIMD) algorithms~\cite{GuilleminRobertZwart:04}.

The principal aim of this work is to provide conditions on $(C,T_{1},T_{2},\ldots)$ such that the solutions to \eqref{eq:sfpe0} exhibit thin tails in the sense that they possess finite exponential moments. More precisely, we will study the domain of the moment generating function (mgf) of a random variable $X$ solving \eqref{eq:sfpe0}, thus
$$ \{\theta\in\R:\Erw e^{\theta X}<\infty\}, $$
and in fact give a precise description of this set in some special cases including random difference equations. Regarding the latter, this will answer a question posed in \cite[Section 4]{AlsIksRoe:09} about the abscissa of convergence of the mgf of $X$ (see results in Section \ref{sec:exponential}), i.e.
\begin{align*}
&r^{*}(X)\ :=\ \sup\{\theta\in\R:\Erw e^{\theta X}<\infty\}
\shortintertext{and}
&r_{*}(X)\ :=\ \inf\{\theta\in\R:\Erw e^{\theta X}<\infty\}.
\end{align*}
Since, by Lemma 3.6 in \cite{BalDalKlu:04},
\begin{align*}
&r^{*}(X)\ :=\ \limsup_{x\to\infty}\frac{-\log\Prob[X>x]}{x}
\shortintertext{and}
&r_{*}(X)\ :=\  \liminf_{x\to\infty}\frac{\log\Prob[X<-x]}{x}
\end{align*}
this gives also information about the rate of exponential decay of the right and left tail of $X$.

Building on observations made by Goldie and Gr\"ubel  \cite{GolGru:96} and later Hitczenko and Weso\l owski \cite{HitczenkoWes:09}, we will also investigate the relation between the tail of $X$ and $\max_{1\le k\le N}T_{k}$. More precisely, we will show that $X$ exhibits Poissonian tails, viz.
$$ \log\Prob[X>x]\ \simeq\ -\frac{\gamma}{\|C^{+}\|_{\infty}}x\log x\quad\text{as }x\to\infty $$
provided that $C$ is a.s. bounded, $T_{1},T_{2},\ldots$ are nonnegative and
\begin{equation*}
\left\|\sum_{k=1}^{N}T_{k}^{p}\right\|_{\infty}\ \le\ 1\quad\text{and}\quad\Prob\left[\max_{1\le k\le N}T_{k}\in (1-\eps,1]\bigg|C>c\right]\ \stackrel{\eps\downarrow 0}{\asymp}\ \eps^{\gamma}
\end{equation*}
for some $p,\gamma>0$ and all $c<\|C^{+}\|_{\infty}$. Here and throughout $\|\cdot\|_{r}$ denotes the usual $L^{r}$-norm for $r\in [1,\infty]$ and
$f(\eps)\stackrel{\eps\downarrow 0}{\asymp}g(\eps)$ means that
$$ 0\ <\ \liminf_{\eps\downarrow 0}\frac{f(\eps)}{g(\eps)}\ \le\ \limsup_{\eps\downarrow 0}\frac{f(\eps)}{g(\eps)}\ <\ \infty. $$

We have organized the paper as follows. Section 2 introduces some basic notation and assumptions. Our main results are stated in Sections 3 and 4, while Section 5 discusses some examples in connection with them. Proofs of the main results are provided in Section 6.

\section{Preliminaries}\label{sec:preliminaries}

Let (the law of) $X$ be a solution to \eqref{eq:sfpe0}. Due to the independence of $(C,T_{1},T_{2},\ldots)$ and $X_{1},X_{2},\ldots$ on the right-hand side the SFPE remains valid under any permutation of the nonzero $T_{k}$. Therefore and since $N$ is a.s. finite, we may assume without loss of generality that
$$ T_{1}\ge T_{2}\ge\ldots\ge T_{N}\quad\text{and}\quad T_{N+1}=T_{N+2}=\ldots =0, $$
so that \eqref{eq:sfpe0} becomes
\begin{equation}\label{eq:sfpe}
X\ \eqdist\ \sum_{k =1}^{N}T_{k}X_{k} +C.
\end{equation}
Since we are dealing with integrable solutions $X$, a replacement of $X, X_{k}$ and $C$ in \eqref{eq:sfpe} with their centerings 
$$ X^{0}:=X-\Erw X,\quad X_{k}^{0}:=X_{k}-\Erw X_{k}\quad\text{and}\quad C^{0}:=C-\Erw X\left(1-\sum_{k=1}^{N}T_{k}\right), $$
respectively, leads to a SFPE of the same type, namely
$$ X^{0}\ \eqdist\ \sum_{k=1}^{N}T_{k}X_{k}^{0}+C^{0}. $$ 
Hence we may w.l.o.g. assume for the rest of this article that
$$ \Erw X\ =\ \Erw C\ =\ 0. $$
The power sums
$$ \Sigma_{\alpha}\ :=\ \sum_{k=1}^{N}|T_{k}|^{\alpha} $$
for $\alpha\in\R_{\geqslant}=[0,\infty)$ will play an important role in our analysis, where $\Sigma_{0}=N$. If the $T_{k}$ are $[-1,1]$-valued, then furthermore
$$ \Sigma_{\infty}\ :=\ \lim_{\alpha\to\infty}\Sigma_{\alpha}\ =\ \sum_{k=1}^{N}\1_{\{|T_{k}|=1\}}, $$
which is one of the leading parameters in the description of the domain of the mgf of a solution to \eqref{eq:sfpe}.

\vspace{.2cm}
Given an infinite-order Ulam-Harris tree $\T=\{\varnothing\}\cup\bigcup_{n\ge 1}\N^{n}$ with weight $T_{k}(v)$ attached to the edge connecting $v$ with $v k$ and weight $C(v)$ attached to the subtree rooted at $v$, suppose that the $(C(v),T_{1}(v),T_{2}(v),\ldots)$ are iid copies of $(C,T_{1},T_{2},\ldots)$. Let also $L(v)$ be the total weight of the branch from the root to $v$ obtained by multiplication of the edge weights along this path. Then put
$$\Sigma_{\alpha,n}\ :=\ \sum_{|v|=n} |L(v)|^{\alpha}$$
for $n \in \N$ , $\alpha\in\R$ and note that $\|\Sigma_{\alpha,n}\|_p = \|\Sigma_{\alpha}\|_{p}^{n}$ for $p \ge 1$. Finally,
$$ N_{n}\ :=\ \# \{v : |v| = n \mbox{ and } L(v) > 0\}\ =\ \Sigma_{0,n}$$
denotes the number of branches of length $n$ with positive weight. Thus $N_{1}\eqdist N$, and with $N_{0} = 1$, the sequence $(N_{n})_{n \ge 0}$ 
forms a simple Galton-Watson process.

Next, with $X(v)$ denoting iid copies of $X$ which are independent of all other occurring random variables, we have (by $n$-fold interation of \eqref{eq:sfpe}) that
\begin{equation}\label{eq:iterates of SFPE}
X\ \eqdist\ \sum_{|v|=n}L(v)X(v)\ +\ \sum_{k=0}^{n-1}\sum_{|v|=k}L(v)C(v)
\end{equation}
for all $n\ge 1$, where $L(\varnothing):=1$. If the second term on the right-hand side of~\eqref{eq:iterates of SFPE}, that is $W_{n-1}:=\sum_{k=0}^{n-1}\sum_{|v|=k}L(v)C(v)$, converges a.s. to
$$ W\ :=\ \sum_{k\ge 0}\sum_{|v|=k}L(v)C(v), $$
then $W$ is also a solution to the SFPE \eqref{eq:sfpe}, called \emph{canonical solution}. Notice that
$$ W_{n-1}\ \eqdist\ \cS^{n}(\delta_{0}) $$
for each $n\in\N$, where $\delta_{0}$ denotes the Dirac measure at 0.
The recursive structure provides us with two useful equations for the $W_{n}$, the first of which being
\begin{equation}\label{eq:forward Wn}
W_{n}\ =\ W_{n-1}\ +\ \sum_{|v|=n}L(v)C(v),
\end{equation} 
called \emph{forward equation}, which is just a consequence of the definition of $W_{n}$. Since $W_{n} \eqdist \cS \cL( W_{n-1})$ we also have
\begin{equation}\label{eq:backward Wn}
W_{n}\ =\ \sum_{k=1}^{N(\varnothing)}T_{k}(\varnothing)W_{n-1}(k)+C(\varnothing),
\end{equation}
called \emph{backward equation}, where $W_{n-1}(1),W_{n-1}(2),...$ are independent copies of $W_{n-1}$ and also independent of $(C(\varnothing),T_{1}(\varnothing),T_{2}(\varnothing),...)$. We further point out that, given $p\in [1,2]$, $(W_{n})_{n\ge 0}$ forms a $L^{p}$-convergent, zero mean martingale if
\begin{equation}\label{eq:standing}
\Erw C\ =\ 0,\quad\|C\|_{p}\ <\ \infty\quad\text{and}\quad\|\Sigma_{p}\|_{1}\ <\ 1.
\end{equation}
The martingale property is easily verified, and the $L^{p}$-convergence follows from
\begin{align*}
\|W_{m+n}-W_{m}\|_{p}\ &\le\ \sum_{k=m+1}^{m+n}\left\|\sum_{|v|=k}L(v)C(v)\right\|_{p}\\
&\le\ 2^{1/p}\|C\|_{p}\sum_{k=m+1}^{m+n}\|\Sigma_{p,k}\|_{p}\\
&\le\ 2^{1/p}\|C\|_{p}\,\frac{\|\Sigma_{p}\|_{p}^{m+1}}{1-\|\Sigma_{p}\|},
\end{align*}
for all $m,n\ge 1$, where for the second line we have used 
\begin{itemize}
\item the \emph{double martingale structure} of $(W_{n})_{n\ge 0}$, first systematically utilized in \cite{AlsRoe:04} for the study of moments of the ordinary Galton-Watson process and later in \cite{AlsKuhl:10} for general weighted branching processes, 
\item the Topchi\u\i-Vatutin inequality for martingales as stated in \cite{AlsRoe:03}, here applied to $\left\|\sum_{|v|=k}L(v)C(v)\right\|_{p}$. 
\end{itemize}
The double martingale structure refers to the fact that each increment of $(W_{n})_{n\ge 0}$, viz.
$$ W_{n}-W_{n-1}\ =\ \sum_{|v|=n}L(v)C(v)\quad (n\ge 1), $$
forms itself a martingale sum when conditioned upon $\sigma(L(v),\,|v|\le n)$.

\vspace{.1cm}
As a particular consequence of \eqref{eq:standing}, $(W_{n})_{n\ge 0}$ is uniformly integrable and thus a Doob martingale, i.e. $W_{n}=\Erw(W|\cF_{n})$ for each $n\ge 0$, where $\cF_{n}=\sigma(W_{0},...,W_{n})$. If $p=2$, we further have
\begin{align*}
\|W_{n}\|_{2}^{2}\ =\ \|W_{n-1}\|_{2}^{2}\ +\ \|\Sigma_{2,n}\|_{1}\|C\|_{2}^{2}
\ =\ \|W_{n-1}\|_{2}^{2}\ +\ \|\Sigma_{2}\|_{1}^{n}\|C\|_{2}^{2}
\end{align*}
for each $n\ge 1$, giving
\begin{equation*}
\sigma_{W}^{2}\ :=\ \|W\|_{2}^{2}\ =\ \sup_{n\ge 0}\|W_{n}\|_{2}^{2}\ =\ \frac{\|C\|_{2}^{2}}{1-\|\Sigma_{2}\|_{1}}\ =\ \frac{\Erw C^{2}}{1-\Erw \Sigma_{2}}.
\end{equation*}
Let us finally point out that the (law of the) canonical solution $W$ is in fact the unique zero-mean fixed point of $\cS$ in $L^{p}$, see R\"osler \cite[Thm.~3]{Roesler:92} for the case $p=2$ and \cite[Thm.~1 and Thm.~3]{Alsmeyer:13} for general $p\in [1,2]$. In view of our standing assumption this means that exponential moments can only exist for this particular solution.

\vspace{.2cm}
We proceed with the introduction of some further notation. The mgf's of $C,W_{n}$ and $W$ are denoted by $\vph,\Psi_{n}$ and $\Psi$ with canonical domains $\bbD_{\vph},\bbD_{\Psi_{n}}$ and $\bbD_{\Psi}$, respectively, thus
$$ \bbD_{\vph}\ :=\ \{\theta\in\R:\Erw e^{\theta C}<\infty\},\quad\text{etc.} $$
We close this section with a basic lemma and note beforehand that, if $\bbD_{\vph}\ne\{0\}$, we may assume w.l.o.g. that $\bbD_{\vph}\cap\R_{\geqslant}\ne\{0\}$, for otherwise the latter holds after switching from $C$ to $-C$ (and thus from $W$ to $-W$).

\begin{Lemma}\label{lem:mgfconv}
Suppose that \eqref{eq:standing} holds for some $p\in [1,2]$. Then
\begin{description}[(b)]
\item[(a)] $\vph(\theta)=\Psi_{0}(\theta)\le\Psi_{1}(\theta)\le...\le\Psi(\theta)$ for all $\theta\in\R$ and thus $\bbD_{\Psi}\subset\bbD_{\vph}$.\vspace{.05cm}
\item[(b)] $\Psi(\theta)=\lim_{n\to\infty}\Psi_{n}(\theta)$ for all $\theta\in\R$.
\end{description}
\end{Lemma}

\begin{proof}
As shown above, condition \eqref{eq:standing} for some $p\in [1,2]$ ensures that $(W_{n})_{n\ge 0}$ is a zero-mean Doob martingale, in particular $\Erw W=0$.
Consequently, $\Psi$ is convex on its domain with unique minumum at 0. Moreover, by using Jensen's inequality and $W_{0}=C(\varnothing)$,
\begin{align*}
\Psi(\theta)\ =\ \Erw e^{\theta W}\ &=\ \Erw\left[\Erw(e^{\theta W}|\cF_{n})\right]\ \ge\ \Erw e^{\theta\Erw(W|\cF_{n})}\ =\ \Erw e^{\theta W_{n}}\ =\ \Psi_{n}(\theta),
\end{align*}
for all $\theta\ge 0$ and $n\ge 0$. Similarly, $\Psi_{n}(\theta)\ge\Psi_{n-1}(\theta)$ can be shown, which is in fact a trivial consequence of the submartingale property of $(e^{\theta W_{n}})_{n\ge 0}$ if the latter sequence is integrable. We have thus proved (a), and (b) then follows because, by an appeal to Fatou's lemma, we also have
$$ \Psi(\theta)\ =\ \Erw\lim_{n\to\infty}e^{\theta W_{n}}\ \le\ \lim_{n\to\infty}\,\Erw e^{\theta W_{n}}\ =\ \lim_{n\to\infty}\Psi_{n}(\theta) \quad\text{for all }\theta\in\R\text{\qed} $$
\end{proof}

\section{Exponential moments}\label{sec:exponential}

The most natural approach to study $\Psi$ is via the functional equation it satisfies as a consequence of the SFPE \eqref{eq:sfpe}. Namely, writing the latter in terms of mgf and conditioning upon $(C,T_{1},T_{2},\ldots)$ leads to
\begin{equation}\label{eq:funsfpe}
\Psi(\theta)\ =\ \Erw\left[e^{\theta C}\prod_{k=1}^{N}\Psi(T_{k}\theta)\right]\quad\text{for }\theta\in\bbD_{\Psi}.
\end{equation}
Bound for a determination of $\bbD_{\Psi}$, Theorem \ref{thm:startingpoint} below constitutes a good starting point because it allows us to focus thereafter on the situation when
\begin{equation}\label{eq:leq1}
\max_{1\le k\le N}|T_{k}|\ \le\ 1\quad\text{a.s.}
\end{equation}
However, we first state the following basic result, in essence due to Goldie and Gr\"ubel \cite{GolGru:96}, about the situation when \eqref{eq:leq1} fails.

\begin{Prop}\label{prop:golgru}
If \eqref{eq:standing} holds for some $p\in [1,2]$ and
$$ \Prob\left[\max_{1\le k\le N}|T_{k}|>1\right]\ >\ 0, $$
then
$$ \Erw e^{\theta|W|}\ =\ \infty $$
for all $\theta\in\R\backslash\{0\}$, thus $\bbD_{\Psi}\cap\R_{\geqslant}=\{0\}$ or $\bbD_{\Psi}\cap\R_{\leqslant}=\{0\}$.
\end{Prop}

\begin{proof}
Since the $T_{1}\ge\ldots\ge T_{N}$, we have $\Prob(T_{1}>1)>0$ or $\Prob(T_{N}<-1)>0$.
Putting $B_{1}:=\sum_{k=2}^{N}T_{k}W_{k}+C$ and $B_{N}:=\sum_{k=1}^{N-1}T_{k}W_{k}+C$, observe that $W$ satisfies the random difference equations
$$ W\ \eqdist\ T_{1}W_{1}+B_{1}\quad\text{and}\quad W\ \eqdist\ T_{N}W_{N}+B_{N} $$
with $(B_{1},T_{1})$, $(B_{N},T_{N})$ being independent of $W_{1}$ and $W_{N}$, respectively.
Now use Theorem 4.1 in \cite{GolGru:96} to infer that
$$ \liminf_{t\to\infty}\frac{\log\Prob[|W|>t]}{\log t}\ >\ -\infty $$
and thus in particular $\Erw e^{\theta|W|}=\infty$ for all $\theta\in\R$.\qed
\end{proof}

Having shown that $\bbD_{\Psi}$ cannot contain an open neighborhood of 0 if \eqref{eq:leq1} fails, Theorem \ref{thm:startingpoint} contains more detailed information for this situation and particularly reveals that $\bbD_{\Psi}$ differs for the cases when the $T_{k}$ are $[-1,1]$-valued or $[0,1]$-valued, more precisely, if
\begin{equation}
\beta\ :=\ \|T_{N}^{-}\|_{\infty}\ =\ \esssup T_{N}^{-}.
\end{equation}
is 0, thus $\Prob[T_{k}>0,k\le N]=1$, or $>0$, thus $\Prob[T_{k}>0,k\le N]<1$,
where $T_{N}=\min_{1\le k\le N}T_{k}$ should be recalled.

\begin{Theorem}\label{thm:startingpoint}
Suppose that \eqref{eq:standing} holds for some $p\in [1,2]$. 
\begin{description}[(b)]
\item[(a)] If $\bbD_{\Psi}\ne\{0\}$, then one of the following four cases occurs:
\end{description}
\begin{description}[(b)]\itemsep2pt
\item[\quad(a1)] \eqref{eq:leq1} holds, $\beta=0$, and $\bbD_{\vph}\ne\{0\}$.
\item[\quad(a2)] \eqref{eq:leq1} holds, $\beta>0$, and $\bbD_{\vph}\supset (-\eps,\eps)$ for some $\eps>0$.
\item[\quad(a3)] $\Prob[T_{1}>1]>0$, $\beta=0$, and $\Prob\left[w^{*}\Sigma_{1}+C\le w^{*}\right]=1$ for some $w^{*}\ge 0$.
\item[\quad(a4)] $\Prob[T_{1}>1]>0$, $\beta=0$, and $\Prob\left[w_{*}\Sigma_{1}+C\ge w_{*}\right]=1$ for some $w_{*}\le 0$.
\end{description}
\begin{description}[(b)]\itemsep3pt
\item[] Moreover, $\bbD_{\Psi}\supset (-\eps,\eps)\cap\bbD_{\vph}$ for some $\eps>0$ in the first two cases, while $\|C^{+}\|_{\infty}\le\|W^{+}\|_{\infty}<\infty$, $\bbD_{\Psi}=\R_{\geqslant}$ must hold in case (a3) and $\|C^{-}\|_{\infty}\le\|W^{-}\|_{\infty}<\infty$, $\bbD_{\Psi}=\R_{\leqslant}$ must hold in case (a4). 
\item[(b)] Conversely, (a3) and (a4) imply $\bbD_{\Psi}\ne\{0\}$, and this is also true for (a1) and (a2) under the additional assumption
\begin{equation}\label{eq:additional for (a1) and (a2)}
\|C\|_{2} <\infty\quad\text{and}\quad\Erw e^{\theta\Sigma_{2}}\ <\ \infty\quad\text{for some }\theta>0.
\end{equation}
\end{description}
\end{Theorem}

\begin{Rem}\label{rem:startingpoint}\rm
The reader should notice that the additional assumption \eqref{eq:additional for (a1) and (a2)} particularly holds if $\|\Sigma_{2}\|_{\infty}<\infty$ or, a fortiori, if \eqref{eq:leq1} and $\|N\|_{\infty}<\infty$ hold. Note further that $W$ has no nontrivial exponential moments whenever $\beta>1$, i.e. $\Prob[T_{N}<-1]>0$.
\end{Rem}

In view of the previous result we will focus hereafter on the situation when \eqref{eq:leq1} holds. As it turns out, $\bbD_{\Psi}$ depends on the behavior of the family
\begin{equation}\label{eq:defndelta}
N_{\delta}\ :=\ \sum_{k=1}^{N}\left(\1_{\{T_{k}>1-\delta\}} + \1_{\{T_{k} <- (1-\delta)\beta \}}\right),\quad\delta\in (0,1).
\end{equation}
Observe that $N_{\delta}$ increases with $\delta$ and
$$ \lim_{\delta\to 0}N_{\delta}\ =\ \sum_{k=1}^{N}\left(\1_{\{T_{k}=1\}}+\1_{\{T_{k} =-\beta \}}\right)\ =:\ \Sigma_{\infty}^{1,-\beta}\quad\text{a.s.} $$
This convergence does not need to be uniform and in general we have
$$ \cN\ :=\ \lim_{\delta\to 0}\|N_{\delta}\|_{\infty}\ \ge\ \|\Sigma_{\infty}^{1,-\beta}\|_{\infty}. $$

The simplest situation occurs when $\cN\le 1$ and is treated in the next two theorems which, in essence, cover those cases where only $T_{1}$ or $|T_{N}|$ can be large in the sense that only one of these entries is allowed to take values arbitrarily close to $\|\max_{1\le k\le N}T_{k}\|_{\infty}$. Theorem \ref{thm:exppos} assumes nonnegative $T_{k}$ $(\beta=0)$ and provides an explicit description of $\bbD_{\Psi}$, while Theorem \ref{thm:expint} deals with the case when $\Prob[T_{N}<0]>0$ $(\beta>0)$. As indicated by Theorem \ref{thm:startingpoint}, this condition causes some asymmetry regarding $\bbD_{\Psi}$ which is encoded in $\beta$. The proofs will be based on the construction of a certain super-solution (a~technique commonly used in the theory of partial differential equations) of the functional equation \eqref{eq:funsfpe} (see Lemmata \ref{prop:oversol}, \ref{lem:lem1} and \ref{lem:lem2}). 

For $\theta\in\R$, we define
\begin{equation}\label{eq:defab}
a(\theta)\,:=\,\Erw e^{\theta C}\1_{\{T_{1}=1\}}\quad\text{and}\quad
b(\theta)\,:=\,\Erw e^{\theta C}\1_{\{T_N=-\beta\}}.
\end{equation}
Also, let $\interior(A)$ denote the interior of the set $A\subset\R$. 


\begin{Theorem}\label{thm:exppos}
Suppose \eqref{eq:standing} for some $p\in [1,2]$, \eqref{eq:leq1}, $\beta=0$ and $\|\Sigma_{2}\|_{\infty}<\infty$. Suppose further that
\begin{equation}\label{eq:big1pos}
\left\|\max_{2\le k\le N}T_{k}\right\|_{\infty}<\ 1\quad\text{and}\quad\Prob[T_{1}=1,N\ge 2]\,=\,0.
\end{equation}
Then
$$ \bbD_{\Psi}\ =\ \bbD_{\vph}\cap\left\{\theta\in\R:a(\theta)<1\right\}. $$
\end{Theorem}

Notice that \eqref{eq:big1pos} and $\beta=0$ entail $N_{\delta}\le 1$ a.s. for all sufficiently small $\delta>0$ and therefore indeed $\cN\le 1$. To ensure the latter when $\beta>0$, a more complicated version of \eqref{eq:big1pos} must be imposed and appears as \eqref{eq:C1}-\eqref{eq:C3} in the subsequent result.

\begin{Theorem}\label{thm:expint}
Suppose \eqref{eq:standing} for some $p\in [1,2]$, \eqref{eq:leq1}, $\beta>0$ and $\|\Sigma_{2}\|_{\infty}<\infty$. Suppose further that, for some $\delta\in (0,1)$,
\begin{align}
\Prob[-\beta(1-\delta)\le T_{k}\le 1-\delta,\,2\le k\le N-1]\,=\,1.\label{eq:C1}\\[.5mm]
\Prob[T_{1}>1-\delta,\,T_{N}<-\beta(1-\delta)]\,=\,0.\label{eq:C2}\\[.5mm]
\Prob[T_{1}=1,N\ge 2]\,=\,\Prob[T_{N}=-\beta,N\ge 2]\,=\,0.\label{eq:C3}
\end{align}
Then the following assertions hold true:
\begin{description}[(b)]\itemsep2pt
\item[(a)] If $\Prob[T_{N}=-\beta]>0$ and $\beta<1$, then
$$ \bbD_{\Psi}\ =\ \left\{\theta:a(-\beta\theta)\vee a(\theta)<1\text{ and }-\beta\theta,\theta\in\bbD_{\vph}\right\}. $$
\item[(b)] If $\Prob[T_{N}=-1]>0$, then
$$ \bbD_{\Psi}\ =\ \left\{\theta:(1-a(-\theta))(1-a(\theta))>b(\theta)b(-\theta)\text{ and }-\theta,\theta\in\bbD_{\vph}\right\}. $$
\item[(c)] If $\Prob[T_{N}=-\beta]=0$, then
$$ \interior(\bbD_{\Psi})\ \subset\ \left\{\theta:a(-\beta\theta)\vee a(\theta)<1\text{ and }-\beta\theta,\theta\in\bbD_{\vph}\right\}\ \subset\ \bbD_{\Psi}. $$
\end{description}
\end{Theorem}

In the random difference case when $N=1$ and thus $T_{1}=T_N$, it is now easy to derive the abscissa of convergence of $\Psi$, viz $r_{*}(W)=\inf\bbD_{\Psi}$ and $r^{*}(W)=\sup\bbD_{\Psi}$, from the previous theorems. The details can be left to the reader.

\section{Poissonian tails}\label{sec:poissonian}

As shown by Theorem \ref{thm:exppos}, the canonical solution $W$ exhibits very thin tails in the sense of possessing exponential moments of any order $(\bbD_{\Psi}=\R)$ if \eqref{eq:standing} for some $p\in [1,2]$, \eqref{eq:leq1}, $\beta=0$, $\|\Sigma_{2}\|_{\infty}<\infty$, $\Prob[T_{1}=1]=0$, and $\bbD_{\vph}=\R$ hold true. It turns out that in this case the tail behavior of $W$ is determined by the behavior of the law of $T_{1}$ in a neighbourhood of $1$, a phenomenon observed for random difference equations by Goldie and Gr\"ubel \cite{GolGru:96} and later by Hitczenko and Weso\l owski \cite{HitczenkoWes:09}.  Note that this relation is further investigated in the upcoming work by Ko\l{}odziejek~\cite{Kolodziejek:15} concerning the random difference equation~\eqref{eq:rde} with $C=1$. In a proper setting, the phenomenon carries over to the canonical fixed point of the smoothing transform. Regarding the right tail of $W$, we will work under the additional assumptions (besides those of Theorem \ref{thm:exppos}) that $C$ is bounded and the law of $T_{1}$, or its conditional law given $C>c$ for any $c\in (0,\|C^{+}\|_{\infty})$, is equivalent to a beta distribution at 1. The first means that, for some $\gamma>0$,
\begin{align}
\label{eq1:beta@1}
&\exists\,\eps,d,D>0: \forall\,\delta\in (0,\eps):\ d\ \le\ \frac{\Prob[1-\delta<T_{1}\le 1]}{\delta^{\gamma}}\ \le\ D
\shortintertext{and the second}
\begin{split}\label{eq2:beta@1}
&\forall\,c\in (0,c^{+}):\,\exists\,\eps,d',D'>0: \forall\,\delta\in (0,\eps):\\
&\hspace{3.05cm} d'\ \le\ \frac{\Prob[1-\delta<T_{1}\le 1|C>c]}{\delta^{\gamma}}\ \le\ D'\\
\end{split}
\end{align}
where $c^{+}:=\|C^{+}\|_{\infty}$. Obviously, \eqref{eq1:beta@1} entails \eqref{eq2:beta@1} if $C$ and $T_{1}$ are independent. Note that~\eqref{eq2:beta@1} implies
$\|C^+ \|_{\infty}= \lim_{\delta \to 0} \| C^+ \1_{\{1-\delta \leq T_{1} \leq 1\}}\|_{\infty}$. Whence, the biggest values of $C$ are attained on the sets where the 
biggest values of $T_{1}$ are attained.

\begin{Theorem}\label{thm:pois}
Given the assumptions of Theorem \ref{thm:exppos}, thus $\beta=0$, suppose further $\|\Sigma_{q}\|_{\infty}\le 1$ for some $q\ge 1$, $\|C^+\|_{\infty}<\infty$, and \eqref{eq1:beta@1} for some $\gamma>0$. Then
\begin{align*}
\limsup_{x\to\infty}\frac{\log\Prob[W>x]}{x\log x}\ \le\ -\frac{\gamma}{c^{+}},
\end{align*}
If, furthermore, \eqref{eq2:beta@1} is valid, then the previous result can be sharpened to
\begin{align*}
\lim_{x\to\infty}\frac{\log\Prob[W>x]}{x\log x}\ =\ -\frac{\gamma}{c^{+}}.
\end{align*}
\end{Theorem}

Since $-W$ is the canonical fixed point of the smoothing transform pertaining to $(-C,T_{1},T_{2},\ldots)$, the corresponding version of the theorem for $\Prob[W<-x]$ can easily be formulated and requires to replace $c^{+}$ with $c^{-}:=\|C^{-}\|_{\infty}$, and also $C$ with $-C$ in \eqref{eq2:beta@1}.

In the case of a random difference equation ($N=1$), Theorem~\ref{thm:pois} improves corresponding results by Goldie and Gr\"ubel \cite{GolGru:96} and Hitczenko and Weso\l owski \cite{HitczenkoWes:09}, for the latter required independence of $C$ and $T_{1}$, while here a dependence is allowed through \eqref{eq2:beta@1}. The result provides us with a general upper bound for $\log\Prob[W>x]$, but if $C$ and $(T_{1},T_{2},\ldots)$ are dependent, this bound does not need to be optimal. Loosely speaking, if such a dependence occurs, the asymptotics depend on the behavior of $C$ on the set $\{T_{1}>1-\delta\}$ for small $\delta$ as made precise by condition \eqref{eq2:beta@1}. A prominent example exhibiting such kind of dependence of $C$ and $T_{1}$ appears in the Quicksort equation \eqref{eq:quicksort} for which a discussion can be found in the next section.

\section{Examples}\label{sec:examples}

We begin with a discussion of the Quicksort distribution in the light of Theorem \ref{thm:pois}. The main result, Eq. \eqref{eq:qsrighttail} below on its right tail, has also been obtained by Janson \cite{Janson:15} in a recent note. He further proved that its left tail shows a very different behavior in being doubly exponential (Gumbel-like).

\begin{Exa}\label{exa:2}\rm
Recall the Quicksort equation \eqref{eq:quicksort}, viz.
$$ X^{(qs)}\ \stackrel{d}{=}\ UX_{1}^{(qs)}+(1-U) X_{2}^{(qs)}+g(U) $$
with unique canonical solution $X^{(qs)}$ having mean 0 and finite variance. Here $U$ has a uniform distribution on $(0,1)$ and $g(t)=2t\log t+2(1-t)\log(1-t)+1$ for $t\in (0,1)$. Obviously, this SFPE fits into our framework with $N=2$, $T_{1}=U\vee (1-U)$, $T_{2}=U\wedge (1-U)$ and 
$$ C\ =\ g(U)\ =\ 2\,T_{1}\log T_{1}+2\,(1-T_{1})\log(1-T_{1})+1 $$
Note also that $\Sigma_{1}=1$, $\Prob[T_{1}>1-\delta]=2\delta$ for $\delta\in (0,1)$, $\|C^{+}\|_{\infty}=1$ and $\|C^{-}\|_{\infty}=2\log 2-1$, where the last two facts follow because
\begin{align*}
\|C^{+}\|_{\infty}\ &=\ \sup_{t\in (0,1)}g(t)\ =\ \lim_{t\uparrow 1}g(t)\ =\ 1,\\
\|C^{-}\|_{\infty}\ &=\ -\inf_{t\in (0,1)}g(t)\ =\ -g(1/2)\ =\ 2\log 2-1.
\end{align*}
The first part of Theorem \ref{thm:pois} therefore provides us with
\begin{align}
\limsup_{x\to\infty}\frac{\log\Prob[X^{(qs)}>x]}{x\log x}\ &\le\ -1\quad\text{and}\label{eq:qsright}\\
\limsup_{x\to\infty}\frac{\log\Prob[X^{(qs)}<-x]}{x\log x}\ &\le\ -\frac{1}{2\log 2-1}\ \approx\ -2.5887.\label{eq:qsleft}
\end{align}
Regarding \eqref{eq2:beta@1}, we have that, for any $c\in (0,1)$ and $t\in (1-\eta_{c},1)$,
\begin{align*}
\Prob[T_{1}>1-t|C>c]\ =\ \frac{\Prob[U\not\in[t,1-t]]}{\Prob[U\not\in [\eta_{c},1-\eta_{c}]]}\ =\ \frac{t}{\eta_{c}},
\end{align*}
where $\eta_{c}$ is the unique value in $[0,\frac{1}{2})$ such that $g(\eta_{c})=c$.
Consequently, \eqref{eq:qsright} can be sharpened to
\begin{equation}\label{eq:qsrighttail}
\lim_{x\to\infty}\frac{\log\Prob[X^{(qs)}>x]}{x\log x}\ =\ -1.
\end{equation}
As already mentioned, the behavior of $\log\Prob[X^{(qs)}<-x]$ is very different and therefore \eqref{eq2:beta@1} must be violated. Indeed, $-C=-g(U)$ attains its maximal values when $U$ is close to $\frac{1}{2}$. As a consequence, $\{T_{1}>1-t\}$ and $\{-C>c\}$ are disjoint and hence
$$ \Prob[T_{1}>1-t|-C>c]\ =\ 0 $$
for all $c$ close to $c^{-}$ and $t$ sufficiently close to 1.
\end{Exa}

\begin{Exa}\label{exa:3} \rm
Very similar to the Quicksort equation \eqref{eq:quicksort} is the median-of-three Quicksort equation \eqref{eq:mtst}, viz. 
$$ X^{(mtqs)}\ \stackrel{d}{=}\ M X_{1}^{(mtqs)} + (1-M)X_{2}^{(mtqs)} + f(M), $$
with $f:(0,1)\to\R$ defined as $f(m):=1+ \frac{12}{7}( m\log(m) + (1-m)\log(1-m))$
and $M=\textsl{med}(U_{1},U_{2},U_{3})$ for independent uniform $(0,1)$ variables $U_{i}$, $i=1,2,3$. Noting that $M$ has a $\beta(1,1)$ distribution with density $6x(1-x)\1_{(0,1)}(x)$ and that $T_{1}=M\vee(1-M)$ satisfies 
$$ \Prob[1-\delta \le T_{1} \leq 1]\ =\ \Prob[M\le\delta]+\Prob[M\ge 1-\delta]\ =\ 6\delta^{2}-3\delta^3 $$
for $0<\delta<1/2$, we find by the same arguments as in Example~\ref{exa:2} that
\begin{align}
\lim_{x\to\infty}\frac{\log\Prob[X^{(mtqs)}>x]}{x\log x}\ &=\ -2\quad\text{and}\label{eq:mtqsright}\\
\limsup_{x\to\infty}\frac{\log\Prob[X^{(mtqs)}<-x]}{x\log x}\ &\le\ -\frac{14}{12\log 2-7}\ \approx\ -10.624.\label{eq:mtqsleft}
\end{align}
We thus see that right and left tails for the normalized number of key comparisons are asymptotically both thinner for the median-of-three version of \texttt{Quicksort} than for its standard counterpart.
\end{Exa}

\begin{Exa}\label{exa:4}\rm
The last example mentioned in the Introduction is the 2-dimensional quad tree equation, viz.
$$ X^{(qt)} \eqdist U_{1}U_{2} X_{1}^{(qt)} + U_{1}(1-U_{2})X_{2}^{(qt)} + (1-U_{1})U_{2} X_{3}^{(qt)} + (1-U_{1})(1-U_{2}) X_4^{(qt)} + h(U_{1},U_{2}), $$
where $h : (0,1)^{2} \to \R$ is defined as
\begin{align*}
h(u_{1},u_{2}) = 1&+ u_{1}u_{2}\log(u_{1}u_{2}) + (1-u_{1})u_{2} \log((1-u_{1})u_{2})\\ 
&+ u_{1}(1-u_{2})\log(u_{1}(1-u_{2})) + (1-u_{1})(1-u_{2})\log((1-u_{1})(1-u_{2}))
\end{align*}
and $U_{1},\,U_{2}$ are iid with uniform distribution on $[0,1]$. Here $\Sigma_{1}=1$, $\Erw\Sigma_{2}=4/9$, $C=h(U_{1},U_{2})$ and
$$ T_{1} = \max\{U_{1}U_{2}, \: U_{1}(1-U_{2}),\: (1-U_{1})U_{2},\: (1-U_{1})(1-U_{2}) \}. $$
Since $\Prob[U_{1}U_{2} \ge 1-\delta] = \frac{\delta^{2}}{2(1-\delta)}$, we have $\Prob[1-\delta \le T_{1}\le 1]=\frac{2\delta^{2}}{1-\delta}$ for $0<\delta<\frac{1}{2}$. 
Furthermore
\begin{align*}
\|h(U_{1},U_{2})^{+}\|_{\infty}\ &=\ \sup_{(u_{1},u_{2})\in (0,1)^{2}}h(u_{1},u_{2})\ =\ \lim_{(u_{1},u_{2})\uparrow (1,1)}h(u_{1},u_{2})\ =\ 1,\\
\|h(U_{1},U_{2})^{-}\|_{\infty}\ &=\ -\inf_{(u_{1},u_{2})\in (0,1)^{2}}h(u_{1},u_{2})\ =\ -h(1/2,1/2)\ =\ 2\log 2-1.
\end{align*}
For any $c\in (0,1)$ pick $\eta_c \in (1/2,1]$ such that $h(\eta_c,\eta_c) = c$ and notice that for $\delta< 1-\eta_c$
$$ \frac{\delta^{2}}{\eta_c^{2}} = \Prob[U_{1}>1-\delta,\,U_{2}>1-\delta] = \Prob[U_{1}>1-\delta,\,U_{2}>1-\delta,C>c ] \leq \Prob[T_{1}>1-\delta,\,C>c]
$$
Having also the upper bound $\Prob[T_{1}>1-\delta,\,C > c]\le\Prob[1-\delta \le  T_{1} \leq 1] = \frac{2\delta^{2}}{1-\delta}$, we arrive at the conclusion that
\begin{align}
\lim_{x\to\infty}\frac{\log\Prob[X^{(qt)}>x]}{x\log x}\ &=\ -2\quad\text{and}\\
\limsup_{x\to\infty}\frac{\log\Prob[X^{(qt)}<-x]}{x\log x}\ &\le\ -\frac{2}{2\log 2-1}.
\end{align}
\end{Exa}

Our next example is to demonstrate that, assuming nonnegative weights $T_{k}$ $(\beta=0)$, information about $\Sigma_{\infty}$, i.e. the number of weights equal to 1, does not suffice to determine $\bbD_{\Psi}$. In some cases we rather need to know the behavior of their laws in small neighbourhoods of 1.

\begin{Exa}\label{exa:1}\rm
Pick any $\alpha<2$ and let $A$ be a random variable with a $\beta(\alpha,1)$ distribution and thus density $\alpha t^{\alpha-1}\1_{(0,1)}(t)$. For any integer $n\ge 2$ satisfying $\alpha<\frac{2}{n-1}$, let further $N\equiv n$, $T_{1}=\ldots=T_{n}=A$ and $C$ be any random variable with mean 0 and independent of $A$. Then \eqref{eq:sfpe} reads
$$ X\ \eqdist\ A\sum_{k=1}^{n}X_{k}+C $$
with associated functional equation \eqref{eq:funsfpe} of the special form
$$ \Psi(\theta)\ =\ \vph(\theta)\int_{0}^{1}\Psi(t\theta)^{n}\alpha t^{\alpha-1}\ dt $$
which in fact allows us to compute $\Psi$ explicitly. By taking derivatives with respect to $\theta$, we obtain
\begin{align*}
\Psi'(\theta)\ &=\ \vph'(\theta)\int_{0}^{1}\Psi(t\theta)^{n}\alpha t^{\alpha-1}\ dt\ +\ \vph(\theta)\int_{0}^{1}nt\,\Psi'(t\theta)\,\Psi(t\theta)^{n-1}\alpha t^{\alpha-1}\ dt\\
&=\ \frac{\vph'(\theta)}{\vph(\theta)}\,\Psi(\theta)\ +\ \frac{\vph(\theta)}{\theta}\int_{0}^{1}\frac{d}{dt}\left[\Psi(t\theta)^{n}\right]\alpha t^{\alpha}\ dt\\
&=\ \frac{\vph'(\theta)}{\vph(\theta)}\,\Psi(\theta)\ +\ \frac{\alpha\,\vph(\theta)\Psi(\theta)^{n}}{\theta}\ -\ \frac{\vph(\theta)}{\theta}\int_{0}^{1}\Psi(t\theta)^{n}\alpha^{2}t^{\alpha-1}\ dt\\
&=\ \frac{\vph'(\theta)}{\vph(\theta)}\,\Psi(\theta)\ +\ \frac{\alpha\,\vph(\theta)\Psi(\theta)^{n}}{\theta}\ -\ \frac{\alpha\,\Psi(\theta)}{\theta}
\end{align*}
and therefore
$$ \Psi'(\theta)\ =\ \frac{\alpha\,\vph(\theta)}{\theta}\Psi(\theta)^{n}\ +\ \left(\frac{\vph'(\theta)}{\vph(\theta)}-\frac{\alpha}{\theta}\right)\Psi(\theta). $$
This is a Bernoulli differential equation and can be solved explicitly. Defining $x(\theta):=\Psi(\theta)^{1-n}$, this function satisfies
$$ 0\ =\ x'(\theta)\ +\ (n-1)\frac{\alpha\vph(\theta)}{\theta}\ +\ (n-1)\left(\frac{\vph'(\theta)}{\vph(\theta)}-\frac{\alpha}{\theta}\right)x(\theta) $$
from which we infer
\begin{align*}
0\ =\ \frac{d}{d\theta}\left[\frac{\vph(\theta)^{n-1}}{\theta^{\alpha(n-1)}}x(\theta)\right]\ +\ (n-1)\frac{\alpha\vph(\theta)^{n}}{\theta^{\alpha(n-1)+1}}
\end{align*}
and thereupon, for any pair $(\theta_{0},\theta)$ with $\theta_{0}<\theta$,
\begin{align*}
\frac{\vph(\theta)^{n-1}}{\theta^{\alpha(n-1)}}x(\theta)\ &=\ \frac{\vph(\theta_{0})^{n-1}}{\theta_{0}^{\alpha(n-1)}}x(\theta_{0})\ -\ (n-1)\int_{\theta_{0}}^{\theta}\frac{\vph(s)^{n}}{s^{\alpha(n-1)+1}}\ ds\\
&=\ \frac{1}{\theta_{0}^{\alpha(n-1)}}(1+o(1))\ -\ (n-1)\int_{\theta_{0}}^{\theta}\frac{\vph(s)^{n}}{s^{\alpha(n-1)+1}}\ ds\\
&=\ \left(\frac{1}{\theta^{\alpha(n-1)}}+\int_{\theta_{0}}^{\theta}\frac{\alpha(n-1)}{s^{\alpha(n-1)+1}}\ ds\right)(1+o(1))\\
&\hspace{3.4cm}-\ (n-1)\int_{\theta_{0}}^{\theta}\frac{\vph(s)^{n}}{s^{\alpha(n-1)+1}}\ ds\\
&=\ \frac{1}{\theta^{\alpha(n-1)}}\ -\ \alpha(n-1)\int_{\theta_{0}}^{\theta}\frac{\vph(s)^{n}-1}{s^{\alpha(n-1)+1}}\ ds\ +\ o(1)
\end{align*}
where the $o(1)$ term is for $\theta_{0}\to 0$ and fixed $\theta$. Finally, by solving for $\Psi(\theta)=x(\theta)^{-1/(n-1)}$ and passing to the limit $\theta_{0}\to 0$, we find
\begin{equation*}
\Psi(\theta)\ =\ \frac{\vph(\theta)}{\left(1-\int_{0}^{\theta}(\vph(s)^{n}-1)\left(\frac{\theta}{s}\right)^{\alpha(n-1)+1}\alpha(n-1)\ ds\right)^{1/(n-1)}}.
\end{equation*}
With this explicit formula for $\Psi$, we see that $\bbD_{\Psi}$ is given by
\begin{equation}\label{eq:exdpsi}
\bbD_{\Psi}\ =\ \bbD_{\vph}\cap\left\{\theta:\int_{0}^{\theta}(\vph(s)^{n}-1)\left(\frac{\theta}{s}\right)^{\alpha(n-1)+1}\alpha(n-1)\ ds<1\right\}
\end{equation}
and thus depends on $\bbD_{\vph}$, the branching index $n$ and, most notably, the parameter $\alpha$ which characterizes the tails of the $T_{k}$ via
\begin{equation}\label{eq:exT}
\Prob[T_{k}>t]\ =\ 1-t^{\alpha}\quad\text{for }t\in (0,1].
\end{equation}
As for $s<\theta$, the function
$$ \alpha\ \mapsto\ \alpha\left(\frac{\theta}{s}\right)^{\alpha(n-1)+1} $$
is increasing, the set $\bbD_{\Psi}$ in \eqref{eq:exdpsi} gets smaller, while the probabilities in \eqref{eq:exT} get bigger with increasing $\alpha$.
\end{Exa}

Our last example shows that the cases (a3) and (a4) in Theorem \ref{thm:startingpoint} can actually occur.

\begin{Exa}\label{exa:cases b and c}\rm
Let $N=2$ and $(C,T_{1},T_{2})$ take values
\begin{equation*}
\left(-1,\frac{5}{4},\frac{1}{4}\right)\quad\text{and}\quad\left(1,\frac{1}{4},\frac{1}{4}\right)
\end{equation*}
with probability $\frac{1}{2}$ each. Then $\vph(\theta)=\cosh\theta\le e^{|\theta|}$ for all $\theta\in\R$, $\|\Sigma_{2}\|_{1}=\frac{7}{8}<1=\|\Sigma_{1}\|_{1}$, and
$$ C+2(\Sigma_{1}-1)\ =\ 0. $$
Obviously in the situation of case (3) in Theorem \ref{thm:startingpoint} with $w^{*}=2$ if $\bbD_{\Psi}\ne\{0\}$, we claim that $\Psi_{n}(\theta)\le e^{2\theta}$ for all $\theta\ge 0$ and $n\in\N_{0}$ which, by Lemma \ref{lem:mgfconv}, entails the same for $\Psi(\theta)$. For an induction over $n$, note that the claim holds for $\vph=\Psi_{0}$. Assuming validity for $\Psi_{n-1}$, the backward equation \eqref{eq:backward Wn} provides us with
\begin{align*}
\Psi_{n}(\theta)\ =\ \Erw\left[e^{\theta C}\prod_{k=1}^{N}\Psi_{n-1}(\theta T_{k})\right]\ \le\ e^{2\theta}\,\Erw e^{\theta(C+2(\Sigma_{1}-1))},
\end{align*}
and since $C+2(\Sigma_{1}-1)=0$ the claim is proved.

\vspace{.1cm}
Fixing any $p\in [1,2]$, we can modify the previous example in such a way that \eqref{eq:standing} holds for this $p$ while $\|C^{-}\|_{\alpha}=\infty$ for any $\alpha>p$. Namely, assume now that $N=2$,
\begin{align*}
&\Prob\left((C,T_{1},T_{2})=\left(1,\frac{1}{4},\frac{1}{4}\right)\right)\ =\ \frac{2}{3},\\
&\Prob\left((T_{1},T_{2})=\left(\frac{5}{4},\frac{1}{4}\right)\right)\ =\ \frac{1}{3},\\
&\Prob\left(C\in\cdot\bigg|(T_{1},T_{2})=\left(\frac{5}{4},\frac{1}{4}\right)\right)\ =\ \Prob(C'\in\cdot),
\end{align*}
where $C'\in L^{p}$ takes values in $(-\infty,-1]$, has mean $-2$ and infinite absolute $\alpha$-moments for $\alpha>p$. Then one can readily verify that $\Erw C=0$, $\vph(\theta)\le e^{\theta}$ for all $\theta\in\R_{\geqslant}$ (as $C\le 1$), $\|\Sigma_{\alpha}\|_{1}<1$ for all $\alpha\in [1,2]$, and
$$ C+2(\Sigma_{1}-1)\ \le\ 0. $$
Therefore the above inductive argument still works to give $\Psi(\theta)\le e^{2\theta}$ for all $\theta\in\R_{\geqslant}$.
\end{Exa}

\section{Proofs}\label{sec:proofs}

Let us begin with a rather simple but useful technical lemma.

\begin{Lemma}\label{prop:oversol}
Suppose \eqref{eq:standing} for some $p\in [1,2]$ and \eqref{eq:leq1}. Let $I\subset\R$ be an interval containing $0$. Then $I\subset\bbD_{\Psi}$ iff there exists a function $\Phi:I\to [1,\infty)$, called super-solution of \eqref{eq:funsfpe} on $I$\footnote{and in fact a superharmonic function for the smoothing transform $\cS$ when viewed as an operator on the halfspace of functions $f:I\to\R_{\geqslant}$ and defined by $\cS f(\theta):=\Erw[e^{\theta C}\prod_{k=1}^{N}f(T_{k}\theta)]$ for $\theta\in I$.}, such that $\Phi(0)=1$ and
\begin{equation}\label{eq:oversol}
\Erw\left[e^{\theta C}\prod_{k=1}^{N}\Phi(T_{k}\theta)\right]\ \le\ \Phi(\theta)\quad\text{for all }\theta\in I.
\end{equation}
In this case $\Psi\le\Phi$ on $I$.
\end{Lemma}

\begin{proof}
Suppose there is a super-solution $\Phi$ and let $W_{-1}:=0$. Then we have
$$ \Psi_{-1}(\theta)\ :=\ \Erw e^{\theta W_{-1}}\ =\ 1\ \le\ \Phi(\theta) $$
for all $\theta\in I$. Now use induction over $n$. Assuming $\Psi_{n-1}\le\Phi$ on $I$,  \eqref{eq:oversol} and the backward equation \eqref{eq:backward Wn}, we obtain
\begin{align*}
\Psi_{n}(\theta)\ =\ \Erw\left[e^{\theta C}\prod_{k=1}^{N}\Psi_{n-1}(T_{k}\theta)\right]\ \le\ \Erw\left[e^{\theta C}\prod_{k=1}^{N}\Phi(T_{k}\theta)\right]\ \le\ \Phi(\theta)
\end{align*}
and therefore, by Lemma \ref{lem:mgfconv},
$$ \Psi(\theta)\ =\ \lim_{n\to\infty}\Psi_{n}(\theta)\ \le\ \Phi(\theta)\ <\ \infty $$
for all $\theta\in I$.

Conversely, if $\bbD_{\Psi}\supset I$, then $\Psi$ itself is a super-solution.\qed
\end{proof}

\begin{proof}[of Theorem \ref{thm:startingpoint}]
(a) Suppose $\bbD_{\Psi}\ne\{0\}$ and pick any $\theta\in\bbD_{\Psi}\backslash\{0\}$. By Lemma \ref{lem:mgfconv}, $\{0\}\ne\bbD_{\Psi}\subset\bbD_{\vph}$, thus (a1) is valid if also \eqref{eq:leq1} and $\beta=0$ are assumed.

Next consider the case when \eqref{eq:leq1} holds, but $\beta>0$, that is $\Prob[T_{N}<0]>0$. Then $\Prob[T_{N}<-\delta]>0$ for some $\delta\in (0,1)$ and hence, by \eqref{eq:funsfpe},
\begin{equation}\label{eq:t<0}
\infty\ >\ \Psi(\theta)\ =\ \Erw\left[e^{\theta C}\prod_{k=1}^{N}\Psi(T_{k}\theta)\right]\ \ge\ \Erw\left[e^{\theta C}\1_{\{T_{N}<-\delta\}}\right]\Psi(-\delta\theta),
\end{equation}
giving $-\delta\theta\in\bbD_{\Psi}$ and thereupon $[-\delta\theta,\delta\theta]\subset\bbD_{\Psi}\subset\bbD_{\vph}$, for $\bbD_{\Psi}$ is convex. In other words, the conditions of (a2) are valid.

\vspace{.2cm}
Finally assume that \eqref{eq:leq1} fails to hold, thus
$$ |T_{1}|\vee|T_{N}|\ =\ \max_{1\le k\le N}|T_{k}|\ >\ 1\quad\text{with positive probability}. $$
This is the most difficult situation and requires some work. Further assuming $\theta>0$, we will show now that the conditions of (a3) are valid. By an analogous argument, those of (a4) follow if $\theta<0$.

\vspace{.2cm}
\textsc{Claim 1}. $\Prob[T_{N}\ge 0]=1$ and thus $\beta=0$.

\vspace{.1cm}\noindent
Assuming the contrary, another use of \eqref{eq:t<0} yields $[-\delta\theta,\delta\theta]\subset\bbD_{\Psi}$, thus $\Erw e^{\delta\theta|W|}<\infty$, which contradicts Proposition \ref{prop:golgru}. Consequently, $T_{N}\ge 0$ a.s. which in turn implies $\beta=0$ and then $T_{1}=\max_{1\le k\le N}T_{k}>1$ with positive probability.

\vspace{.2cm}
\textsc{Claim 2}. $\bbD_{\Psi}=\R_{\geqslant}$.

\vspace{.1cm}\noindent
Choose $\eps>0$ such that $\gamma:=\Prob[T_{1}>1+\eps]>0$. By another use of \eqref{eq:funsfpe}, we infer that
\begin{equation}\label{eq:funt>1}
\infty\ >\ \Psi(\theta)\ =\ \Erw\left[e^{\theta C}\prod_{k=1}^{N}\Psi(T_{k}\theta)\right]\ \ge\ \Erw\left[e^{\theta C}\1_{\{T_{1}>1+\eps\}}\right]\Psi((1+\eps)\theta)
\end{equation}
and thus $(1+\eps)\theta\in\bbD_{\Psi}$. Iterating this argument, we obtain $\R_{\geqslant}\subset\bbD_{\Psi}$. By another appeal to Proposition \ref{prop:golgru}, we must have $\bbD_{\Psi}=\R_{\geqslant}$.

\vspace{.2cm}
\textsc{Claim 3}. $\Prob[C\le 0|T_{1}>1]=1$.

\vspace{.1cm}\noindent
If $\Prob[C>0|T_{1}>1]>0$, then \eqref{eq:funt>1} remains valid with $\Erw[e^{\theta C}\1_{\{T_{1}>1\}}]\Psi(\theta)>0$ on the right-hand side and we arrive at the impossible conclusion that
$$ 1\ \ge\ \Erw\left[e^{\theta C}\1_{\{T_{1}>1\}}\right]\ \ge\ \Erw\left[e^{\theta C}\1_{\{C\ge 0,T_{1}>1\}}\right]\ \stackrel{\theta\to\infty}{\longrightarrow}\ \infty $$
(having used $\bbD_{\Psi}=\R_{\geqslant}$).

\vspace{.2cm}
\textsc{Claim 4}. $W$ is a.s. bounded from above, i.e. $\|W^{+}\|_{\infty}<\infty$.

\vspace{.1cm}\noindent
Assuming the contrary, i.e. $\|W^{+}\|_{\infty}=\infty$, it is a well-known fact that $\log\Psi$ is an increasing strictly convex function on its domain $\bbD_{\Psi}=\R_{\geqslant}$, whence its derivative $\Psi'(\theta)/\Psi(\theta)$ increases to $\infty$ as $\theta\to\infty$. As a consequence,
\begin{align*}
\frac{1}{\eps\theta}\log\left(\frac{\Psi((1+\eps)\theta)}{\Psi(\theta)}\right)\ =\ \frac{\log\Psi((1+\eps)\theta)-\log\Psi(\theta)}{\eps\theta}\ \ge\ \frac{\Psi'(\theta)}{\Psi(\theta)}\ \stackrel{\theta\to\infty}{\longrightarrow}\ \infty
\end{align*}
for any fixed $\eps>0$. In other words, $\Psi((1+\eps)\theta)=e^{\theta h(\theta)}$ for some function $h$ satisfying $\lim_{\theta\to\infty}h(\theta)=\infty$. Now let $\eps$ and $\gamma$ be as under Claim 2 and use $\nu:=\Erw[C|T_{1}>1+\eps]\le 0$ by Claim 3 in combination with Jensen's inequality to infer
$$ \Erw\left[e^{\theta C}\1_{\{T_{1}>1+\eps\}}\right]\ \ge\ \gamma\,\Erw\left[e^{\theta C}|T_{1}>1+\eps\right]\ \ge\ \gamma\,e^{\theta\nu} $$
for all $\theta\ge 0$. Returning to \eqref{eq:funt>1} and using the previous facts, we arrive at the contradiction
\begin{align*}
1\ \ge\ \frac{\Psi((1+\eps)\theta)}{\Psi(\theta)}\,\Erw\left[e^{\theta C}\1_{\{T_{1}>1+\eps\}}\right]\ \ge\ e^{\theta(\nu+h(\theta))}\ \stackrel{\theta\to\infty}{\longrightarrow}\ \infty.
\end{align*}

\vspace{.2cm}
\textsc{Claim 5}. $\Prob[w^{*}\Sigma_{1}+C\le w^{*}]=\Prob[C\le w^{*}]=1$ for $w^{*}:=\|W^{+}\|_{\infty}$.

\vspace{.1cm}\noindent
Since $(C,T_{1},T_{2},\ldots)$ and $(W,W_{1},W_{2},\ldots)$ are independent and all $T_{k}$ are nonnegative $(\beta=0)$, the SFPE \eqref{eq:sfpe} provides us with
\begin{align*}
w^{*}\ =\ \esssup\left(\sum_{k=1}^{N}T_{k}W_{k}+C\right)\ =\ \esssup\left(w^{*}\Sigma_{1}+C\right)
\end{align*}
which in turn implies
$$ 1\ =\ \Prob[w^{*}\Sigma_{1}+C\le w^{*}]\ \le\ \Prob[C\le w^{*}] $$
as claimed.

\vspace{.2cm}
(b) It remains to show that each of the cases (a1)-(a4), the first two under the additional assumption \eqref{eq:additional for (a1) and (a2)}, implies $\bbD_{\Psi}\ne\{0\}$ and that even $\bbD_{\Psi}\supset (-\eps,\eps)\cap\bbD_{\vph}$ for some $\eps>0$ holds true under (a1) and (a2).

\vspace{.1cm}
If (a3) holds, then $w^{*}\Sigma_{1}+C\le w^{*}$ a.s. for some $w^{*}\ge 0$ entails $C\le w^{*}$ a.s. because all $T_{k}$ are nonnegative. By using the backward equation \eqref{eq:backward Wn} inductively, we then obtain $W_{n}\le w^{*}$ a.s. and thereupon $W\le w^{*}$ a.s. which in turn implies $\bbD_{\Psi}\supset\R_{\geqslant}$. A similar argument shows $\bbD_{\Psi}\supset\R_{\leqslant}$ if (a4) is valid.

Left with the cases (a1) and (a2), which are treated together, we first note that in case (a1) we may assume w.l.o.g. that $\bbD_{\vph}\cap\R_{>}\ne\emptyset$, for otherwise we may switch to the smoothing transform based on $(-C,T_{1},T_{2},\ldots)$ and with canonical fixed point $-W$. We further note that \eqref{eq:standing} for some $p\in [1,2]$ combined with \eqref{eq:leq1} entails $\|\Sigma_{2}\|_{1}<1$. Recall that $\sigma_{W}^{2}=\Var W=\Erw C^{2}/(1-\Erw\Sigma_{2})$ is finite. For $\theta\in\R$, consider now the random function
$$ G(\theta)\ :=\ e^{\theta C+b\theta^{2}(\Sigma_{2}-1)}, $$
where the constant $b>\sigma_{W}^{2}/2$ is chosen in such a way that
$$ \Erw C^{2}+b(\Erw\Sigma_{2}-1)\ <\ 0. $$
The first three derivatives of $G$ with respect to $\theta$ are given by
\begin{align*}
G'(\theta)\ &=\ \big(C+2\theta b(\Sigma_{2}-1)\big)G(\theta),\\
G''(\theta)\ &=\ \left(\big(C+2\theta b(\Sigma_{2}-1)\big)^{2}C+2b(\Sigma_{2}-1)\right)G(\theta),\\
G'''(\theta)\ &=\ \big(4b(\Sigma_{2}-1)+1\big)\big(C+2b\theta(\Sigma_{2}-1)\big)G(\theta),
\end{align*}
so that
$$ G(0)\,=\,1,\quad G'(0)\,=\,C\quad\text{and}\quad G''(0)\,=\,C^{2}+2b(\Sigma_{2}-1). $$
By \eqref{eq:additional for (a1) and (a2)}, we can fix $\theta_{0}\in\bbD_{\vph}\cap\R_{>}$ sufficiently small such that, with the help of H\"older's inequality,
$$ \Erw G(2\theta_{0})\ \le\ \vph(4\theta_{0})^{1/2}\left(\Erw e^{8b\theta_{0}^{2}(\Sigma_{2}-1)}\right)^{1/2}\ <\ \infty. $$
For $\theta\in (-\theta_{0},\theta_{0})$, we then obtain
$$ G'''(\theta)\ \le\ \big|4b(\Sigma_{2}-1)+1\big|\big(|C|+2b\theta_{0}|\Sigma_{2}-1|)e^{\theta_{0}C^{+}+b\theta_{0}^{2}(\Sigma_{2}-1)}\ =:\ \ovl{G}(\theta_{0}), $$
and $\Erw\ovl{G}(\theta_{0})<\infty$. A third-order Taylor expansion of $\Erw G(\theta)$ about 0  provides us with
\begin{align*}
\Erw G(\theta)\ &\le\ G(0)\ +\ \theta\,\Erw G'(0)\ +\ \frac{\theta^{2}}{2}\,\Erw G''(0)\ +\ \frac{|\theta|^{3}}{6}\,\Erw\ovl{G}(\theta_{0})\\
&=\ 1\ +\ \frac{\theta^{2}}{2}\,\left(\Erw C^{2}+2b(\Erw\Sigma_{2}-1)\right)\ +\ \frac{|\theta|^{3}}{6}\,\Erw\ovl{G}(\theta_{0})
\end{align*}
for all sufficiently small $\theta\in\bbD_{\vph}$. By the choice of $b$, we can now fix $\delta>0$ such that for any $\theta\in I:=(-\delta,\delta)\cap\bbD_{\vph}$, we have
$$ \Erw G(\theta)\le 1. $$
But this implies that the function $\Phi:I\to [1,\infty]$, $\Phi(\theta):=e^{b\theta^{2}}$ satisfies condition \eqref{eq:oversol} of Lemma \ref{prop:oversol} and thus leads to the conclusion that $I\subset\bbD_{\Psi}$ as asserted.\qed
\end{proof}

Notice that form the last proof, $\Psi(\theta) \leq e^{b\theta^{2}}$ for sufficiently small $\theta \in \mathbb{D}_{\Psi}$. Since this inequality in also valid for 
$\theta$ bounded away from $0$ by increasing $b$, we may infer that for any $I \subseteq \mathbb{D}_{\Psi}$ one can always pick $B_I$ large enough such that 
$\Psi(\theta)\leq e^{B_I\theta^{2}}$ for $\theta\in I$.

For the proofs of Theorem \ref{thm:exppos} and Theorem \ref{thm:expint}, the main work is provided by two subsequent lemmata so as to keep the presentation as transparent as possible.

\begin{Lemma}\label{lem:lem1}
Suppose $\beta=1$ and the assumptions of Theorem \ref{thm:expint} be satisfied. Then, for any $\theta\in\R$, $-\theta,\theta\in\bbD_{\Psi}$ iff
$$ (1-a(\theta))(1-a(-\theta))\,>\,b(\theta)b(-\theta)\quad\text{and}\quad [-\theta,\theta]\,\subset\,\bbD_{\vph}. $$
\end{Lemma}

\begin{proof}
Put
$$ \wh{\bbD}_{\Psi}\ :=\ \{\theta\in\R:-\theta,\theta\in\bbD_{\Psi}\} $$
and
$$ \bbD\ :=\ \{\theta\in\R:(1-a(\theta))(1-a(-\theta))>b(\theta)b(-\theta)\text{ and } -\theta,\theta\in\bbD_{\vph}\}, $$
so that $\wh{\bbD}_{\Psi}=\bbD$ must be verified. Note that both sets are symmetric about 0 by definition. By Theorem \ref{thm:startingpoint}, $\bbD\ne\{0\}$ and thus $0\in\interior(\bbD_{\vph})$ entails $\bbD_{\Psi}\supset (-\eps,\eps)$ for some $\eps>0$ (Case (a2) there).

For the inclusion $\wh{\bbD}_{\Psi}\subset\bbD$, let $\theta\in\wh{\bbD}_{\Psi}$. Then $[-\theta,\theta]\subset\bbD_{\vph}$ by an appeal to Lemma \ref{lem:mgfconv}. Using the functional equation \eqref{eq:funsfpe}, we find
\begin{align*}
\Psi(\theta)\ &\ge\ \Erw\left[e^{\theta C}\1_{\{T_{1}=1\}}\right]\Psi(\theta)\ +\ \Erw\left[e^{\theta C}\1_{\{T_{N}=-1\}}\right]\Psi(-\theta)+ \Erw \left[e^{\theta C}\1_{\{|T_{1}|,|T_N|<1 \}}\right]\\
&>\ \Erw\left[e^{\theta C}\1_{\{T_{1}=1\}}\right]\Psi(\theta)\ +\ \Erw\left[e^{\theta C}\1_{\{T_{N}=-1\}}\right]\Psi(-\theta)
\end{align*}
as well as
$$ \Psi(-\theta)\ >\ \Erw\left[e^{-\theta C}\1_{\{T_{1}=1\}}\right]\Psi(-\theta)\ +\ \Erw\left[e^{-\theta C}\1_{\{T_{N}=-1\}}\right]\Psi(\theta), $$
and these inequalities may be rewritten as
\begin{align*}
\Psi(\theta)(1-a(\theta))\ &>\ b(\theta)\Psi(-\theta)
\shortintertext{and}
\Psi(-\theta)(1-a(-\theta))\ &>\ b(-\theta)\Psi(\theta),
\end{align*}
respectively. Since all appearing quantities are positive and finite, multiplication yields
$$ \Psi(\theta)\Psi(-\theta)(1-a(\theta))(1-a(-\theta))\ >\ \Psi(-\theta)\Psi(\theta)b(\theta)b(-\theta) $$
and thus $\theta\in\bbD$.

\vspace{.2cm}
Having just shown $\wh{\bbD}_{\Psi}\subset\bbD$, suppose now that this inclusion is strict, i.e. $\bbD\backslash\wh{\bbD}_{\Psi}\ne\emptyset$, in particular $\bbD\ne\{0\}$ and thus 
\begin{equation}\label{eq:theta_{0} positive}
\bbD_{\Psi}\supset\wh{\bbD}_{\Psi}\supset (-\eps,\eps)\quad\text{for some }\eps>0. 
\end{equation}
By our assumptions, there exists $\delta_{1}>0$ such that
$$ \left\|\max_{2\le k\le N}T_{k}^{+}\right\|_{\infty}\vee\left\|\max_{1\le k\le N-1}T_{k}^{-}\right\|_{\infty}\ <\ 1-\delta_{1}. $$
Define $\theta_{0}:=\inf\R_{>}\cap(\bbD\backslash\wh{\bbD}_{\Psi})$, which is positive by \eqref{eq:theta_{0} positive}. Moreover, it follows that $[-(1-\delta_{1})\theta_{0},(1-\delta_{1})\theta_{0}]\subset\wh{\bbD}_{\Psi}\cap\interior(\bbD_{\Psi})$ and
\begin{equation*}
\frac{1-a(\theta_{0})}{b(\theta_{0})}\ >\ \eta\ >\ \frac{b(-\theta_{0})}{1-a(-\theta_{0})}
\end{equation*}
for some $\eta>0$. For $\delta>0$, consider
$$ a_{\delta}(\theta_{0})\ :=\ \Erw\left[e^{\theta_{0}C}\prod_{k=2}^{N}\Psi(T_{k}\theta_{0})\1_{\{T_{1}\in (1-\delta,1]\}}\right]. $$
Since $(1-\delta_{1})\theta_{0}\in\interior(\bbD_{\Psi})$, we have that $\log\Psi(\theta)\le \kappa\theta^{2}$ for all $\theta\in [0,(1-\delta_{1})\theta_{0}]$ and some $\kappa=\kappa(\theta_{0})>0$. As a consequence,
$$ e^{\theta_{0}C}\prod_{k=2}^{N}\Psi(T_{k}\theta_{0})\1_{\{T_{1}\in (1-\delta,1]\}}\ \le\ e^{\theta_{0}C+\kappa\theta_{0}^{2}\Sigma_{2}}\quad\text{a.s.} $$
and thereby (using $\|\Sigma_{2}\|_{\infty}<\infty$ and $\bbD_{\Psi}\subset\bbD_{\vph}$)
$$ \Erw\left[e^{\theta_{0}C+\kappa\theta_{0}^{2}\Sigma_{2}}\right]\ \le\ \vph(\theta_{0})\,e^{\kappa\theta_{0}^{2}\|\Sigma_{2}\|_{\infty}}<\ \infty. $$
With the help of the dominated convergence theorem, we now infer that
$$ a_{\delta}(\theta_{0})\ \stackrel{\delta\to 0}{\longrightarrow}\ a(\theta_{0})\ <\ 1, $$
and by a similar argument also
$$  b_{\delta}(\theta_{0})\ :=\ \Erw\left[e^{\theta_{0}C}\prod_{k=1}^{N-1}\Psi(T_{k}\theta_{0})\1_{\{T_{N}\in [-1,-1+\delta)\}}\right]\ \stackrel{\delta\to 0}{\longrightarrow}\ b(\theta_{0}) $$
and the corresponding assertions with $-\theta_{0}$ instead of $\theta_{0}$. Therefore, we can pick $\delta\in (0,\delta_{1})$ such that
\begin{equation*}
\frac{1-a_{\delta}(\theta_{0})}{b_{\delta}(\theta_{0})}\ >\ \eta\ >\ \frac{b_{\delta}(-\theta_{0})}{1-a_{\delta}(-\theta_{0})},
\end{equation*}
and then by continuity further $\theta_{*}\in\wh{\bbD}_{\Psi}$ and $\theta^{*}\in\bbD\backslash\wh{\bbD}_{\Psi}$ such that $(1-\delta)\theta^{*}<\theta_{*}<\theta_{0}\le\theta^{*}$ and
\begin{equation*}
\frac{1-a_{\delta}(\theta)}{b_{\delta}(\theta)}\ >\ \eta\ >\ \frac{b_{\delta}(-\theta)}{1-a_{\delta}(-\theta)}\quad\text{for all }\theta\in [\theta_{*},\theta^{*}].
\end{equation*}
Let $\delta$ also be small enough to guarantee \eqref{eq:C1} and \eqref{eq:C2} of Theorem \ref{thm:expint}.

\vspace{.1cm}
Consider the function $\Phi:[-\theta^{*},\theta^{*}]\to [1,\infty)$, defined by
\begin{equation}\label{eq:defphi}
\Phi(\theta)\ :=\ 
\begin{cases}
\hfill \xi\eta,&\text{if }\theta\in [-\theta^{*},-\theta_{*}),\\
\Psi(\theta),&\text{if }\theta\in [-\theta_{*},\theta_{*}],\\
\hfill \xi,&\text{if }\theta\in (\theta_{*},\theta^{*}]
\end{cases}
\end{equation}
for some $\xi\ge\Psi(\theta_{*})\vee\eta^{-1}\Psi(-\theta^{*})$. As will be shown next, $\xi$ can be chosen in such a way that $\Phi$ satisfies \eqref{eq:oversol}. To this end we point out first that $\Phi=\Psi$ on $[-\theta_{*},\theta_{*}]$ in combination with \eqref{eq:leq1} ensures
$$ \Erw\left[e^{\theta C}\prod_{k=1}^{N}\Phi(T_{k}\theta)\right]\ =\ \Erw\left[e^{\theta C}\prod_{k=1}^{N}\Psi(T_{k}\theta)\right]\ =\ \Psi(\theta)\ =\ \Phi(\theta) $$
for all $\theta\in [-\theta_{*},\theta_{*}]$ whence we need to verify \eqref{eq:oversol} for $|\theta|\in(\theta_{*},\theta^{*}]$. Put
$$ M\,:=\,\left\{\max_{1\le k\le N}|T_{k}|\le 1-\delta\right\}\quad\text{and}\quad c_{\delta}(\theta)\,:=\,\Erw\left[e^{\theta C}\prod_{k=1}^{N}\Phi(T_{k}\theta)\1_{M}\right] $$
and note that by \eqref{eq:C1} and \eqref{eq:C2}, if $T_{1}>1-\delta$, then $|T_{k}|<1-\delta$ a.s. for $k\ne 1$, thus
$$ T_{1}>1-\delta\quad\RA\quad T_{k}\theta\in [-\theta_{*},\theta_{*}]\quad\text{a.s. for }|\theta|\le|\theta^{*}|\text{ and }2\le k\le N, $$
where $(1-\delta)\theta^{*}<\theta_{*}$ should be recalled. By an analogous argument,
$$ T_{N}<-1+\delta\quad\RA\quad T_{k}\theta\in [-\theta_{*},\theta_{*}]\quad\text{a.s. for }|\theta|\le|\theta^{*}|\text{ and }1\le k\le N-1. $$
With these observations, we arrive at the inequality
\begin{align*}
\Erw\left[e^{\theta C}\prod_{k=1}^{N}\Phi(T_{k}\theta)\right]\ &\le\ \Phi(\theta)\,\Erw\left[e^{\theta C}\prod_{k=2}^{N}\Phi(T_{k}\theta)\1_{\{T_{1}>1-\delta\}}\right]\\
&+\ \Phi(-\theta)\,\Erw\left[e^{\theta C}\prod_{k=1}^{N-1}\Phi(T_{k}\theta)\1_{\{T_{N}<-1+\delta\}}\right]\\
&+\ \Erw\left[e^{\theta C}\prod_{k=1}^{N}\Phi(T_{k}\theta)\1_{M}\right]\\
&=\ \Phi(\theta)a_{\delta}(\theta)+\Phi(-\theta)b_{\delta}(\theta)+c_{\delta}(\theta),
\end{align*}
valid for $\theta\in [-\theta^{*},\theta^{*}]$. So we must verify that $\xi$ can be chosen in such a way that, for $|\theta|\in (\theta_{*},\theta^{*}]$,
\begin{align}\label{eq:boundforfg}
\Phi(\theta)a_{\delta}(\theta)+\Phi(-\theta)b_{\delta}(\theta)+c_{\delta}(\theta)\ \le\ \Phi(\theta).
\end{align}
or, equivalently,
\begin{align}
\xi a_{\delta}(\theta)+\xi\eta b_{\delta}(\theta)+c_{\delta}(\theta)\ &\le\ \xi\quad\text{and}\label{eq:forA11}\\
\xi\eta a_{\delta}(-\theta)+\xi b_{\delta}(-\theta)+c_{\delta}(-\theta)\ &\le\ \xi\eta\label{eq:forA12}
\end{align}
for $\theta\in (\theta_{*},\theta^{*}]$. For \eqref{eq:forA11}, this is obviously requires
$$ \xi\ \ge\ \sup_{\theta\in [-\theta_{*},\theta^{*}]}\frac{c_{\delta}(\theta)}{(1-a_{\delta}(\theta))-\eta b_{\delta}(\theta)}, $$
while \eqref{eq:forA12} requires
$$ \xi\ \ge\ \sup_{\theta\in [-\theta_{*},\theta^{*}]}\frac{c_{\delta}(-\theta)}{\eta(1-a_{\delta}(-\theta))-b_{\delta}(-\theta)}. $$
Since both suprema are positive and finite, we can choose $\xi$ to be smallest number in $[\ge\Psi(\theta_{*})\vee\eta^{-1}\Psi(-\theta^{*}),\infty)$ satisfying both inequalities.
Then $\Phi$ defined by \eqref{eq:defphi} satisfies \eqref{eq:oversol} of Lemma \ref{prop:oversol}, whence this lemma implies $[-\theta^{*},\theta^{*}]\subset\bbD_{\Psi}$, i.e. $\theta^{*}\in\wh{\bbD}_{\Psi}$, which is a contradiction.\qed
\end{proof}

The proof of the next lemma differs from the previous one only in some technical aspects and we therefore supply details only where necessary.

\begin{Lemma}\label{lem:lem2}
Suppose $\beta<1$ and the assumptions of Theorem \ref{thm:exppos} $(\beta=0)$ or Theorem \ref{thm:expint} $(\beta>0)$ be satisfied. Then, for any $\theta\in\R$, $-\beta\theta,\theta\in\bbD_{\Psi}$ iff
$$ a(\theta)\vee a(-\beta\theta)\,<\,1\quad\text{and}\quad -\beta\theta,\theta\,\in\,\bbD_{\vph}. $$
\end{Lemma}

\begin{proof}
Here we put
$$ \wh{\bbD}_{\Psi}\ :=\ \{\theta\in\R:-\beta\theta,\theta\in\bbD_{\Psi}\} $$
and
$$ \bbD\ :=\ \{\theta\in\R:a(\theta)\vee a(-\beta\theta)<1\text{ and } -\beta\theta,\theta\in\bbD_{\vph}\}, $$
and must again verify $\wh{\bbD}_{\Psi}=\bbD$.

\vspace{.1cm}
For the proof of $\wh{\bbD}_{\Psi}\subset\bbD$, pick any $\theta\in\wh{\bbD}_{\Psi}$. Then \eqref{eq:funsfpe} provides us with
\begin{align*}
\Psi(\theta)\ &>\ \Erw\left[e^{\theta C}\1_{\{T_{1}=1\}}\right]\Psi(\theta)\ +\ \Erw\left[e^{\theta C}\1_{\{T_{N}=-\beta\}}\right]\Psi(-\beta\theta)
\shortintertext{and}
\Psi(-\beta\theta)\ &>\ \Erw\left[e^{-\beta\theta C}\1_{\{T_{1}=1\}}\right]\Psi(-\beta\theta)\ +\ \Erw\left[e^{\theta C}\1_{\{T_{N}=-\beta\}}\right]\Psi(\beta^{2}\theta),
\end{align*}
which in turn lead to
\begin{align*}
\Psi(\theta)(1-a(\theta))\ >\ 0\quad\text{and}\quad\Psi(-\beta\theta)(1-a(-\beta\theta))\ >\ 0,
\end{align*}
respectively. This obviously proves the asserted inclusion.

\vspace{.1cm}
Assuming this inclusion to be proper, thus $\bbD\backslash\wh{\bbD}_{\Psi}\ne\emptyset$ and $\bbD\ne\{0\}$, we infer validity of \eqref{eq:theta_{0} positive} as in the previous lemma by an appeal to Theorem \ref{thm:startingpoint} (Cases (a1) or (a2)). Note also that, by our assumptions, there exists $\delta_{1} \in (0,1-\beta)\neq \emptyset$ such that
$$ \left\|\max_{2\le k\le N}T_{k}^{+}\right\|_{\infty}<\ 1-\delta_{1}\quad\text{and}\quad\left\|\max_{1\le k\le N-1}T_{k}^{-}\right\|_{\infty}\ \le\ (1-\delta_{1})\beta. $$
Once again, $\theta_{0}:=\inf\R_{>}\cap(\bbD\backslash\wh{\bbD}_{\Psi})$ is positive by \eqref{eq:theta_{0} positive}, and we have further that $[-\beta(1-\delta_{1})\theta_{0},(1-\delta_{1})\theta_{0}]\subset\wh{\bbD}_{\Psi}\cap\interior(\bbD_{\Psi})$ and
$$ 1-a(\theta_{0})\ >\ \eta b(\theta_{0}) $$
for some $\eta>0$. As argued in the previous proof,
\begin{align*}
&a_{\delta}(\theta_{0})\ =\ \Erw\left[e^{\theta_{0}C}\prod_{k=2}^{N}\Psi(T_{k}\theta_{0})\1_{\{T_{1}\in (1-\delta,1]\}}\right]\ \stackrel{\delta\to 0}{\longrightarrow}\ a(\theta_{0})\ <\ 1,\\
\shortintertext{and}
&b_{\delta}(\theta_{0})\ :=\ \Erw\left[e^{\theta_{0}C}\prod_{k=2}^{N}\Psi(T_{k}\theta_{0})\1_{\{T_{N}\in [-\beta,-(1-\delta)\beta)\}}\right]\ \stackrel{\delta\to 0}{\longrightarrow}\ b(\theta_{0}).
\end{align*}
Therefore, we can pick $\delta\in (0,\delta_{1})$ such that (notice the difference here to the previous proof)
\begin{equation*}
1-a_{\delta}(\theta_{0})\,>\,\eta b_{\delta}(\theta_{0})\quad\text{and}\quad a_{\delta}(-\beta\theta_{0})\,<\,1,
\end{equation*}
and then $\theta_{*}\in\wh{\bbD}_{\Psi}$ and $\theta^{*}\in\bbD\backslash\wh{\bbD}_{\Psi}$ such that $(1-\delta)\theta^{*}<\theta_{*}<\theta_{0}\le\theta^{*}$ and
\begin{equation*}
1-a_{\delta}(\theta)\,>\,\eta b_{\delta}(\theta)\quad\text{and}\quad a_{\delta}(-\beta\theta)\,<\,1\quad\text{for all }\theta\in [\theta_{*},\theta^{*}].
\end{equation*}
Again, let $\delta$ also be small enough to guarantee \eqref{eq:C1} and \eqref{eq:C2} of Theorem \ref{thm:expint} if $\beta>0$, and $\|\max_{2\le k\le N}T_{k}\|_{\infty}<1-\delta$ if $\beta=0$. Note that since $\beta<1-\delta<1-\delta_{1}<1$, then $\beta^{2}\theta^*< \theta_*$.

\vspace{.1cm}
Defining the function $\Phi:[-\theta^{*},\theta^{*}]\to [1,\infty)$ by
\begin{equation*}
\Phi(\theta)\ :=\ 
\begin{cases}
\hfill \xi\eta,&\text{if }\theta\in [-\beta\theta^{*},-\beta\theta_{*}),\\
\Psi(\theta),&\text{if }\theta\in [-\beta\theta_{*},\theta_{*}],\\
\hfill \xi,&\text{if }\theta\in (\theta_{*},\theta^{*}]
\end{cases}
\end{equation*}
for some $\xi\ge\Psi(\theta_{*})\vee\eta^{-1}\Psi(-\beta\theta^{*})$, the remaining proof follows almost the same lines as the previous one with lines~\eqref{eq:boundforfg},~\eqref{eq:forA11} and~\eqref{eq:forA12} replaced by
\begin{align}
\Phi(\theta) a_{\delta}(\theta) + \Phi(-\beta\theta) b_{\delta}(\theta) +c_{\delta}(\theta) \ & \le\ \Phi(\theta)\label{eq:forA2}\\
\xi a_{\delta}(\theta) + \eta \xi b_{\delta}(\theta) + c_{\delta}(\theta) \ &\le\ \xi \label{eq:forA21}\\
\eta \xi a_{\delta}(-\beta\theta) + \Psi(\theta_*) b_{\delta}(-\beta\theta) +c_{\delta}(-\beta\theta) \ & \le\ \eta \xi \label{eq:forA22}
\end{align}
respectively. We arrive at the same conclusion that, for suitable $\xi$, $\Phi$ satisfies \eqref{eq:oversol} of Lemma \ref{prop:oversol}, thus producing the contradiction $[-\beta\theta^{*},\theta^{*}]\subset\bbD_{\Psi}$, i.e. $\theta^{*}\in\wh{\bbD}_{\Psi}$. Further details are therefore omitted.\qed
\end{proof}

\begin{proof}[of Theorem \ref{thm:exppos}]
This result now follows directly from Lemma \ref{lem:lem2} for the case $\beta=0$.\qed
\end{proof}

\begin{proof}[of Theorem \ref{thm:expint}]
Here a separate discussion of the two cases $\Prob[T_{N}=-\beta]>0$ and $\Prob[T_{N}=-\beta]=0$ is necessary.

\vspace{.1cm}
If the first alternative occurs, then $b(\theta)=\Erw[e^{\theta C}\1_{\{T_{N}=-\beta\}}]>0$ for any $\theta\in\bbD_{\Psi}$. By another appeal to \eqref{eq:funsfpe}, we then infer
$$ \infty\ >\ \Psi(\theta)\ \ge\ \Psi(-\beta\theta)\,b(\theta)\ >\ 0 $$
for any $\theta\in\bbD_{\Psi}$ and thereby $-\beta\theta,\theta\in\bbD_{\Psi}$. The assertion now follows from Lemma \ref{lem:lem1} if $\beta=1$, and from Lemma \ref{lem:lem2} if $\beta<1$.

\vspace{.1cm}
If $\Prob[T_{N}=-\beta]=0$, then \eqref{eq:funsfpe} provides us with
$$ \infty\ >\ \Psi((1-\eps)^{-1}\theta)\ \ge\ \Psi(-\beta\theta)\,\Erw\left[e^{\theta C}\1_{\{T_{N}\in (-\beta,-(1-\eps)\beta]\}}\right]\ >\ 0 $$
for any $\theta\in\interior(\bbD_{\Psi})$ and some $\eps\in (0,1)$ such that $(1-\eps)^{-1}\theta\in\bbD_{\Psi}$, thus giving $-\beta\theta\in\bbD_{\Psi}$. The assertion finally follows as before from Lemma \ref{lem:lem1} if $\beta=1$, and from Lemma \ref{lem:lem2} if $\beta<1$.\qed
\end{proof}

\begin{proof}[of Theorem \ref{thm:pois}]
The proof consists of two steps, establishing the upper bound and, assuming~\eqref{eq2:beta@1}, the lower bound. Note that, under the given assumptions,
$\R_{\geqslant}\subseteq \bbD_{\vph}$, $\Prob[T_{1}=1]=0$ and thus $a(\theta)=0$ for all $\theta\in\R$. Consequently, by Theorem \ref{thm:exppos}, $\R_{\geqslant}\subseteq\bbD_{\Psi}$.

\vspace{.2cm}
\textsc{Upper bound.} Consider the function $\Phi:\R_{\geqslant}\to [1,\infty)$, defined by
\begin{align*}
\Phi(\theta)\ :=\ 
\begin{cases}
\hfill\Psi(\theta),&\text{if }\theta\in [0,1],\\
\exp(\xi\theta^{q}e^{b\theta}),&\text{if }\theta\in (1,\infty),
\end{cases}
\end{align*}
with $b:=c^{+}/\gamma$, $q$ such that $\|\Sigma_{q}\|_{\infty}\le 1$, and $\xi>0$. We claim that $\xi$ can be chosen so large such that $\Phi$ is a supersolution of \eqref{eq:funsfpe} on $\R_{\geqslant}$, i.e., satisfies \eqref{eq:oversol} on this set. Since this is plain for $\theta\le 1$, we must only consider $\theta>1$. Put
$$ \zeta\,:=\,\sup_{\theta\in [0,1]}\frac{\log\Psi(\theta)}{\theta^{2}}. $$
As $\Phi(\theta)\le\exp(\zeta\theta^{2}+\xi\theta^{q}e^{b\theta})$ for all $\theta\ge 0$, it suffices to verify
$$ \Erw\left[\exp\left(\theta C+\zeta\theta^{2}\Sigma_{2}+\xi\theta^{q}\sum_{k=1}^{N}T_{k}^{q}e^{bT_{k}\theta}-\xi\theta^{q}e^{b\theta}\right)\right]\ \le\ 1\quad\text{for }\theta>1. $$
For any $t\in (0,1]$, the expectation on the right-hand side is bounded by
\begin{align}
\begin{split}\label{eq:multline}
&\Erw\left[\1_{\{T_{1}>t\}}\exp\left(\theta c^{+}+\zeta\Sigma_{2}+\xi\theta^{q}\sum_{k=2}^{N}T_{k}^{q}e^{bT_{k}\theta}\right)\right]\\
&\qquad +\ \Erw\left[\1_{\{T_{1}\le t\}}\exp\left(\theta c^{+}+\zeta\Sigma_{2}+\xi\theta^{q}\Sigma_{q}e^{bT_{1}\theta}-\xi\theta^{q}e^{b\theta}\right)\right].
\end{split}
\end{align}
Pick $\rho\in (0,1/2D)$ with $\gamma,\,D$ as in \eqref{eq1:beta@1} and put
$$ t\ =\ t(\theta,\rho)\ :=\ 1-\left(\frac{\rho}{D}\right)^{1/\gamma}e^{-\theta c^{+}/\gamma}. $$
Then, by using the assumptions of the theorem, the first term in \eqref{eq:multline} can be bounded by
\begin{align*}
&D(1-t)^{\gamma}\exp\left(\theta c^{+}+\zeta\|\Sigma_{2}\|_{\infty}+\xi\theta^{q}(1-t^{q})\exp\left(b(1-t^{q})^{1/q}\theta\right)\right)\\
&\qquad=\ \rho\,\exp\left(\zeta\|\Sigma_{2}\|_{\infty}+\xi\theta^{q}(1-t^{q})\exp\left(b(1-t^{q})^{1/q}\theta\right)\right).
\end{align*}
Since $1-t^{q}\le e^{-\theta c^{+}q/\gamma}$, we further see that
$$ \zeta\|\Sigma_{2}\|_{\infty}+\xi\theta^{q}(1-t^{q})\exp\left(b(1-t^{q})^{1/q}\theta\right)\ \le\ c_{1}\ <\ \infty\quad\text{for all }\theta>1 $$
and thus obtain, by choosing $\rho$ sufficiently small,
$$ \Erw\left[\1_{\{T_{1}>t\}}\exp\left(\theta c^{+}+\zeta\Sigma_{2}+\xi\theta^{q}\sum_{k=2}^{N}T_{k}^{q}e^{bT_{k}\theta}\right)\right]\ \le\ \rho\,e^{c_{1}}\ <\ \frac{1}{2} $$
for $\theta>1$. Here the reader should notice that the choice of $\rho$ depends on the value of $\xi$ (through $c_{1}$) which is still to be chosen.

\vspace{.1cm}
Turning to the second term in \eqref{eq:multline}, it therefore remains to verify that, uniformly in $\rho\in (0,1/2D)$,
$$ \Erw\left[\1_{\{T_{1}\le t\}}\exp\left(\theta c^{+}+\zeta\Sigma_{2}+\xi\theta^{q}\Sigma_{q}e^{bT_{1}\theta}-\xi\theta^{q}e^{b\theta}\right)\right]\ \le\ \frac{1}{2} $$
for $\theta>1$, $t$ as chosen above, and a suitable $\xi>0$. For this to be true, one can take $\xi$ such that
$$ \sup_{\theta>1}\,\exp\left(\theta c^{+}+\zeta\|\Sigma_{2}\|_{\infty}+\xi\theta^{q}e^{b\theta}\left(e^{-b\theta(1-t)}-1\right)\right)\ \le\ \frac{1}{2}, $$
(recall $\|\Sigma_{q}\|_{\infty}\le 1$), and this is indeed possible because (regardless of the value of $\rho$)
$$ b\theta(1-t)\ =\ b\theta e^{-b\theta}\left(\frac{\rho}{D}\right)^{1/\gamma}\ \le\ b\theta e^{-b\theta}\ \le\ e^{-1}\ \le\ \log 2 $$
and $e^{-u}-1\le-u/2$ for $0<u<\log 2$, giving
\begin{align*}
&\theta c^{+}+\zeta\|\Sigma_{2}\|_{\infty}+\xi\theta^{q}e^{b\theta}\left(e^{-b\theta(1-t)}-1\right)\\
&\qquad\le\ \theta c^{+}+\zeta\|\Sigma_{2}\|_{\infty}-e^{b\theta}\frac{\xi b\,\theta^{q+1}}{2}\\
&\qquad=\ \theta c^{+}+\zeta\|\Sigma_{2}\|_{\infty}-\frac{\xi b\,\theta^{q+1}}{2}
\end{align*}
with the last bound going to 0 as $\xi\to\infty$, uniformly in $\theta\ge 1$.

\vspace{.1cm}
Having shown that $\Phi$ satisfies \eqref{eq:oversol} on $\R_{\geqslant}$, Lemma \ref{prop:oversol} provides us with 
$$ \Erw e^{\theta W}\ =\ \Psi(\theta)\ \le\ \Phi(\theta)=\ \exp\left(\xi\theta^{q}e^{b\theta}\right)\quad\text{for all }\theta>1. $$
Picking an arbitrary $\eps\in(0,1)$, there exists $\theta_{0}>1$ such that
$$  \exp\left(\xi\theta^{q}e^{b\theta}\right)\ \le\ e^{(b+\eps)\theta}\text{for all }\theta\ge\theta_{0}. $$
As a consequence, we obtain
\begin{equation*}
\Prob[W>x]\ \le\ e^{-\theta x}\,\Erw e^{\theta W}\ \le\ \exp\left(e^{(b+\eps)\theta}-\theta x\right)
\end{equation*}
for $\theta>\theta_{0}$ and $x>(b+1)e^{\theta_{0}(b+1)}$. The minimum in $\theta$ on the right-hand side occurs at $\theta=\frac{1}{b+\eps}\log\frac{x}{b+\eps}$, and with this $\theta$ we obtain
\begin{align*}
\frac{\log\Prob[W>x]}{x\log x}\ \le\ -\frac{1}{b+\eps}+\frac{1+\log(b+\eps)}{\log x}
\end{align*}
and thereupon
$$ \limsup_{x\to\infty}\frac{\log\Prob[W>x]}{x\log x}\ \le\ -\frac{1}{b}\ =\ -\frac{\gamma}{c^{+}}, $$
since $\eps\in (0,1)$ was arbitrary.

\vspace{.2cm}
\textsc{Lower bound}. Now assume additionally that \eqref{eq2:beta@1} holds. Recall from \eqref{eq:defndelta} that since $\beta=0$
$$ N_{\eps}\ =\ \sum_{k=1}^{N}\1_{\{T_{k}>1-\eps\}} $$
which decreases to $N_{0}=\Sigma_{\infty}$ as $\eps\to 0$. Under the given assumptions, we can fix $\eps_{0}\in (0,1)$ so small that 
\begin{itemize}\itemsep2pt
\item \eqref{eq2:beta@1} is valid with $\eps=\eps_{0}$,
\item $z:=\Prob[\{T_{1}\le 1-\eps_{0},C\ge 0\} \cup \{0\le C\le c^{+}-\eps_{0}\}]>0$,
\item $\|\max_{2\le k\le N}T_{k}\|_{\infty}\le 1-\eps_{0}$, and 
\item $0<\Prob[T_{1}>1-\eps_{0}]<1$, hence $N_{\eps_{0}}\le 1$ a.s. and $0<\Erw N_{\eps_{0}}<1$. 
\end{itemize}
In the associated weighted branching model as specified in Section \ref{sec:preliminaries}, define the homogeneous stopping line (see \cite[Section 7]{AlsBigMei:12} for the general definition)
$$ \cT_{\eps}\ :=\ \{vk:T_{k}(v)\le 1-\eps\text{ or }C(v)\le c,\,\forall\,uj\prec v:C(u)>c,\,T_{j}(u)>1-\eps\}. $$
for $\eps\in (0,\eps_{0})$ and $c\in (c^{+}-\eps_{0},c^{+})$. Then the SFPE \eqref{eq:sfpe} then implies
\begin{equation}\label{eq:sfpestopline}
W\ \eqdist\ \sum_{v\in\cT_{\eps}}L(v)W(v)\ +\ \sum_{v\prec\cT_{\eps}}L(v)C(v).
\end{equation}
With $e_{k}:=(1,1,\ldots,1)\in\N^{k}$ for $k\in\N$ and $e_{0}:=\varnothing$, define the stopping time, consider the stopping time
$$ \tau\ :=\ \inf\{k\ge 0:T_{1}(e_{k})\le 1-\eps\text{ or }C(e_{k})\le c\} $$
along the leftmost path in the given weighted branching tree. Plainly, $\tau$ has a geometric distribution, viz.
$$ \Prob[\tau=k]\ =\ \Prob[T_{1}\le 1-\eps\text{ or }C\le c]\,\Prob[T_{1}>1-\eps,\,C>c]^{k}\quad\text{for }k\ge 0, $$
and we note for later use that, by \eqref{eq2:beta@1},
\begin{align}
\begin{split}\label{eq:later use}
\Prob[\tau=k,\,C(e_{k})\ge 0]\ &=\ \Prob[\{T_{1}\le 1-\eps,C\ge 0\}\cup\{0\le C\le c\}]\\
&\qquad\times\,\Prob[T_{1}>1-\eps,\,C>c]^{k}\\
&\ge\ z(\kappa d'\eps^{\gamma})^{k}\quad\text{for }k\ge 0,
\end{split}
\end{align}
where $\kappa:=\Prob[C>c]>0$.
Next, observe that
\begin{align*}
\cT_{\eps}\ &=\ \{e_{\tau+1}\}\cup\bigcup_{j=0}^{\tau}\{e_{k}j:1\le j\le N(e_{k})\},\\
&\{v\prec\cT_{\eps}\}\ =\ \{e_{k}:0\le k\le\tau\},
\end{align*}
and $C(v)>c$ for $k<\tau$, hence
\begin{equation}\label{eq:lower bound}
\sum_{v\prec\cT_{\eps}}L(v)C(v)\ \ge\ \sum_{k=0}^{\tau-1}L(e_{k})C(e_{k})\ >\ \sum_{j=0}^{\tau-1}(1-\eps)^{j}c\ =\ \frac{1-(1-\eps)^{\tau}}{\eps}\,c.
\end{equation}
Define the event
$$ \fX\ :=\ \left\{W(v)\ge 0\text{ for all }v\in\cT_{\eps}\right\}, $$
and further $r:=\Prob[W\ge 0]>0$ and $f(s):=\Erw[s^{N}]$ for $s\in [0,1]$. On the event $\{\tau=k,\,C(e_{k})\ge 0\}\cap\fX$, we have
\begin{align*}
&\hspace{1.7cm}|\{v<\cT_{\eps}\}|\ =\ k+1,\quad |\cT_{\eps}|\ \le\ 1+\sum_{j=0}^{k-1}N(e_{j}),
\shortintertext{and}
&\sum_{v\in\cT_{\eps}}L(v)W(v)\ +\ \sum_{v\prec\cT_{\eps}}L(v)C(v)\ >\ \sum_{j=0}^{k-1}(1-\eps)^{j}c\ =\ \frac{1-(1-\eps)^{k}}{\eps}\,c.
\end{align*}
Moreover, using \eqref{eq:later use},
\begin{align*}
\Prob[\{\tau=k,\,C(e_{k})\ge 0\}\cap\fX]\ &\ge\ z\kappa\,d'\eps^{k}\,\Erw\left[r^{1+\sum_{j=0}^{k-1}N(e_{j})}\right]\\
&=\ zr(\kappa d'\eps^{\gamma})^{k}f(r)^{k}
\end{align*}
and therefore, in view of \eqref{eq:sfpestopline} and \eqref{eq:lower bound},
$$ \Prob\left[\frac{W}{c}>\frac{1-(1-\eps)^{k}}{\eps}\right]\ \ge\ zr(\kappa d'\eps^{\gamma})^{k}f(r)^{k} $$
for all $k\in\N_{0}$. Setting $a:=\kappa d'f(r)$, this further yields
\begin{equation}\label{eq:prob estimate}
\Prob\left[\frac{W}{c}>\frac{1-(1-\eps)^{y}}{\eps}\right]\ \ge\ zr(a\eps^{\gamma})^{y+1}
\end{equation}
for all $y\in\R_{\geqslant}$. Now let $x\ge c$. Then we may choose, for some $\delta\in (0,\eps)$,
$$ \eps\,:=\,\frac{\delta c}{x}\quad\text{and}\quad y\,:=\,\frac{\log(1-\delta)}{\log(1-\eps)}\,=\,\frac{\log(1-\delta)}{\log(1-\delta c/x)} $$
to further infer from \eqref{eq:prob estimate}
\begin{align*}
\log\Prob[W>x]\ &\ge\ \log(zr)\ +\ \left(\frac{\log(1-\delta)}{\log(1-\delta c/x)}+1\right)\log(a\eps^{\gamma})\\
&=\ \log(zr)\ -\ \left(\frac{\log(1-\delta)}{\log(1-\delta c/x)}+1\right)\left(\gamma\log x-\log(a(\delta c)^{\gamma})\right)
\end{align*}
Keeping $\delta$ fixed and letting $x$ tend to $\infty$, we have $\log(1-\delta c/x)\simeq -\delta c/x$ and so
\begin{equation*}
\liminf_{x\to\infty}\frac{\log\Prob[W>x]}{x\log x}\ \ge\ \frac{\gamma\log(1-\delta)}{\delta c}\  \ge\ \frac{\gamma\log(1-\delta)}{\delta(c^{+}-\eps_{0})}.
\end{equation*}
This being true for any fixed $\eps_{0}$ and $\delta$ sufficiently small, we finally arrive at the desired conclusion
\begin{equation*}
\liminf_{x\to\infty}\frac{\log\Prob[W>x]}{x\log x}\ \ge\ -\frac{\gamma}{c^{+}}
\end{equation*}
by first letting $\delta\downarrow 0$, giving $\frac{\log(1-\delta)}{\delta}\to-1$, and then $\eps_{0}\downarrow 0$ (which implies $c\to c^{+}$).\qed

\end{proof}

\section*{Acknowledgment}

The main part of this work was done while the second author was visiting the Institute of Mathematical Statistics at M\"unster in September 2014 and 
February 2015. He gratefully acknowledges financial support and hospitality.

\bibliographystyle{abbrv}
\bibliography{StoPro}

\end{document}